 \documentclass[conference, 10pt]{IEEEtran}

\IEEEoverridecommandlockouts                              

\usepackage{graphics} 
\usepackage{epsfig} 
\usepackage{mathptmx} 
\usepackage{mathrsfs}
\usepackage{times} 
\usepackage{amsmath} 

\usepackage{amsfonts}
\usepackage{amsthm}
\usepackage{amssymb}  
\usepackage{wasysym}
\usepackage{txfonts}
\let\labelindent\relax
\usepackage{enumitem}
\usepackage{bbm}
\usepackage{bm}
\usepackage[singlelinecheck=false,font=footnotesize]{caption}
\usepackage{subcaption}
\usepackage{color}  
\usepackage[dvipsnames]{xcolor}
\usepackage{indentfirst}
\usepackage{multirow}
\usepackage{color, colortbl}
\usepackage{algorithm}
\usepackage{algpseudocode}

\usepackage{geometry}
\usepackage[
    style=ieee,
    doi=false,
    isbn=false,
    url=false,
    eprint=false,
    backend=bibtex,
    natbib=true
    ]{biblatex}

\bibliography{references}

\newtheorem{defn}{Definition}
\newtheorem{rem}[defn]{Remark}

\newtheorem{assum}[defn]{Assumption}

\newtheorem{thm}[defn]{Theorem}
\newtheorem{cor}[defn]{Corollary}

\providecommand{\R}{\ensuremath \mathbb{R}}
\providecommand{\N}{\ensuremath \mathbb{N}}

\providecommand{\tfin}{t_\regtext{f}}

\providecommand{\S}{\ensuremath \mathcal{S}}

\providecommand{\x}{\mathbf{x}}
\providecommand{\y}{\mathbf{y}}
\providecommand{\z}{\mathbf{z}}
\providecommand{\uu}{\mathbf{u}}
\providecommand{\LL}{\mathcal{L}}
\providecommand{\LLbounds}{\mathcal{L}_{\regtext{bounds}}}
\providecommand{\LLsave}{\mathcal{L}_{\regtext{save}}}
\providecommand{\LLelim}{\mathcal{L}_{\regtext{elim}}}
\providecommand{\B}{\mathcal{B}}
\providecommand{\A}{\mathcal{A}}
\newcommand{\norm}[1]{\left\Vert#1\right\Vert}

\newcommand{\defemph}[1]{\emph{#1}}

\newcommand{\ceil}[1]{\left\lceil#1\right\rceil}

\newcommand{\ts}[1]{\textsuperscript{#1}}

\newcommand{\regtext}[1]{\mathrm{\textnormal{#1}}}

\newcommand{\hi}{_\regtext{hi}}

\newcommand{\des}{_\regtext{des}}

\definecolor{Gray}{gray}{0.9}
\newcolumntype{g}{>{\columncolor{Gray}}c}

\newcommand{\fmincon}{\texttt{fmincon}}
\newcommand{\Nineq}{\alpha}
\newcommand{\Neq}{\beta}

\newcommand{\blue}[1]{{\color{blue} #1}}

\newcommand{\eeq}{\epsilon_\regtext{eq}}

\newcommand{\pu}{p_{\regtext{up}}^*}
\newcommand{\pl}{p_{\regtext{lo}}^*}

\newcommand{\aha}{{\hat{\alpha}}}

\newcommand{\itmx}{(\x, B_p(\x), B_{gi}(\x), B_{hj}(\x))}
\newcommand{\itmy}{(\y, B_p(\y), B_{gi}(\y), B_{hj}(\y))}

\algblock{ParFor}{EndParFor}
\algnewcommand\algorithmicparfor{\textbf{parfor}}
\algnewcommand\algorithmicpardo{\textbf{do}}
\algnewcommand\algorithmicendparfor{\textbf{end\ parfor}}
\algrenewtext{ParFor}[1]{\algorithmicparfor\ #1\ \algorithmicpardo}
\algrenewtext{EndParFor}{\algorithmicendparfor}

\usepackage{hyperref}

\title{\LARGE \bf Safe, Optimal, Real-time Trajectory Planning with a Parallel Constrained Bernstein Algorithm}
\author{Shreyas Kousik$^*$, Bohao Zhang$^*$, Pengcheng Zhao$^*$, and Ram Vasudevan
\thanks{* These authors contributed equally to this work.}
\thanks{This work is supported by the Ford Motor Company via the Ford-UM Alliance under award N022977, and the Office of Naval Research under award number N00014-18-1-2575.
The authors are with the Department of Mechanical Engineering, University of Michigan, Ann Arbor, MI. {\tt\small <skousik,jimzhang,pczhao,ramv>@umich.edu}.}%
}
\begin{document}

\maketitle
\thispagestyle{empty}
\pagestyle{plain}

\begin{abstract}
To move through the world, mobile robots typically use a receding-horizon strategy, wherein they execute an old plan while computing a new plan to incorporate new sensor information.
A plan should be dynamically feasible, meaning it obeys constraints like the robot's dynamics and obstacle avoidance; it should have liveness, meaning the robot does not stop to plan so frequently that it cannot accomplish tasks; and it should be optimal, meaning that the robot tries to satisfy a user-specified cost function such as reaching a goal location as quickly as possible.
Reachability-based Trajectory Design (RTD) is a planning method that can generate provably dynamically-feasible plans.
However, RTD solves a nonlinear polynmial optimization program at each planning iteration, preventing optimality guarantees; furthermore, RTD can struggle with liveness because the robot must brake to a stop when the solver finds local minima or cannot find a feasible solution.
This paper proposes RTD*, which certifiably finds the globally optimal plan (if such a plan exists) at each planning iteration.
This method is enabled by a novel Parallelized Constrained Bernstein Algorithm (PCBA), which is a branch-and-bound method for polynomial optimization.
The contributions of this paper are: the implementation of PCBA; proofs of bounds on the time and memory usage of PCBA; a comparison of PCBA to state of the art solvers; and the demonstration of PCBA/RTD* on a mobile robot.
RTD* outperforms RTD in terms of optimality and liveness for real-time planning in a variety of environments with randomly-placed obstacles.
\end{abstract}

\section{Introduction}\label{sec:introduction}

\begin{figure}[t]
    \includegraphics[width=\columnwidth]{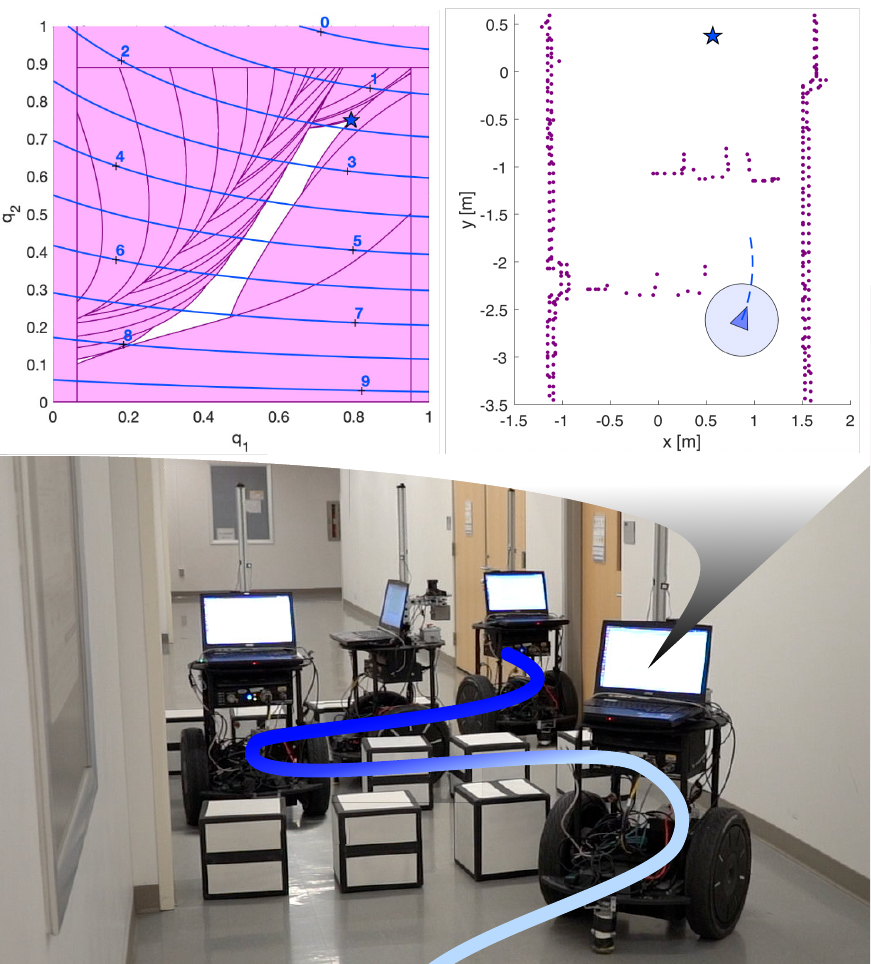}
    \caption{A Segway RMP mobile robot using the proposed PCBA/RTD* method to autonomously navigate a tight obstacle blockade.
    The executed trajectory is shown fading from light to dark blue as time passes, and the robot is shown at four different time instances.
    The top right plot shows the Segway's (blue circle with triangle indicating heading) view of the world at one planning iteration, with obstacles detected by a planar lidar (purple points).
    The top left plot shows the optimization program solved at the same planning iteration; the decision variable is $(q_1, q_2)$, which parameterizes the velocity and yaw rate of a trajectory plan; the pink regions are infeasible with respect to 
    constraints generated by the obstacle points in the right plot; and the blue contours with number labels depict the cost function, which is constructed to encourage the Segway to reach a waypoint (the star in the top right plot).
    The optimal solution found by PCBA is shown as a star on the left plot, in the non-convex feasible area (white).
    This optimal solution generates a provably-safe trajectory for the Segway to track, shown as a blue dashed line in the right plot.
    A video is available at \texttt{\url{https://youtu.be/YcH4WAzqPFY}}.}
    \label{fig:time_lapse}
\end{figure}

For mobile robots to operate successfully in unforeseen environments, they must plan their motion as new sensor information becomes available.
This \defemph{receding-horizon} strategy requires iteratively generating a plan while simultaneouly executing a previous plan.
Typically, this requires solving an optimization program in each planning iteration (see, e.g., \cite{kuwata_phd_receding_horizon}).

The present work considers a mobile ground robot tasked with reaching a global goal location in an arbitrary, static environment.
To assess receding-horizon planning performance, we consider the following characteristics of plans.
First, a plan should be \emph{dynamically feasible}, meaning that it obeys the dynamic description of the robot and obeys constraints such as actuator limits and obstacle avoidance.
Second, a plan should maintain \emph{liveness}, meaning that it keeps the robot moving through the world without stopping frequently to replan, which can prevent a robot from achieving a task in a user-specified amount of time.
Third, a plan should be \emph{optimal} with respect to a user-specified cost function, such as reaching a goal quickly.

Ensuring that plans have these characteristics is challenging for several reasons.
First, robots typically have nonlinear dynamics; this means that creating a dynamically-feasible plan typically requires solving a nonlinear program (NLP) at runtime.
However, it is difficult to certify that an NLP can be solved in a finite amount of time, meaning that the robot may have to sacrifice liveness.
Furthermore, even if a robot has linear dynamics, the cost function may be nonlinear; in this case, it is challenging to certify optimality due to the presence of local minima.

This paper extends prior work on Reachability-based Trajectory Design (RTD).
RTD is able to provably generate \emph{dynamically-feasible} trajectory plans in real time, but cannot guarantee optimality or liveness (the robot will generate plans in real time, but may brake to a stop often).
We address this gap by proposing the \textbf{Parallel Constrained Bernstein Algorithm (PCBA)}, which provably finds globally-optimal solutions to Polynomial Optimization Problems (POPs), a special type of NLP.
We apply PCBA to RTD to produce an optimal version of RTD, which we call \textbf{RTD*} in the spirit of the well-known RRT* algorithm \cite{karaman_rrt_star}.
We show on hardware that RTD* demonstrates \emph{liveness} (in comparison to RTD) for trajectory optimization.
For the remainder of this section, we discuss related work, then state our contributions.

\subsection{Related Work}\label{sec:intro:related_work}

\subsubsection{Receding-horizon Planning}
A variety of methods exist that attempt receding-horizon planning while maintaining dynamic feasibility, liveness, and optimality.
These methods can be broadly classified by whether they perform sampling or solve an optimization program at each planning iteration.
Sampling-based methods typically either attempt to satisfy liveness and dynamic feasibility by choosing samples offine \cite{pivtoraiko_state_lattice,majumdar_funnel_library}, or attempt to satisfy optimality at the potential expense of liveness and dynamic feasibility \cite{karaman_rrt_star}.
Optimization-based methods attempt to find a single optimal trajectory.
These methods typically have to sacrifice dynamic feasibility (e.g., by linearizing the dynamics) to ensure liveness \cite{muske_linear_mpc}, or sacrifice liveness to attempt to satisfy dynamic feasibility \cite{yi_NMPC,dfk2019_fastrack} (also see \cite[\S 9]{kousik2018bridging} and \cite{gpopsii}).

\subsubsection{Reachability-based Trajectory Design}
Reachability-based Trajectory Design (RTD) is an optimization-based receding horizon planner that requires solving a POP at each planning iteration \cite{kousik2018bridging}.
RTD specifies plans as parameterized trajectories.
Since these trajectories cannot necessarily be perfectly tracked by the robot, RTD begins with an offline computation of a Forward Reachable Set (FRS).
The FRS contains every parameterized plan, plus the tracking error that results from the robot not tracking any plan perfectly.
At runtime, in each planning iteration, the FRS is intersected with sensed obstacles to identify all parameterized plans that could cause a collision (i.e., be dynamically infeasible).
This set of unsafe plans is represented as a (finite) list of polynomial constraints, and the user is allowed to specify an arbitrary (not necessarily convex) polynomial cost function, resulting in a POP.
At each planning iteration, either the robot successfully solves the POP to get a new plan, or it continues executing its previously-found plan.
While the decision variable is typically only two- or three-dimensional, each POP often has hundreds of constraints, making it challenging to find a feasible solution in real-time \cite{kousik2018bridging}.
Each plan includes a braking maneuver, ensuring that the robot can always come safely to a stop if the POP cannot be solved quickly enough in any planning iteration.

Note, for RTD, \defemph{optimality} means finding the optimal solution to a POP at each planning iteration.
The cost function in RTD's POPs typically encode behavior such as reaching a waypoint between the robot's current position and the global goal (e.g., \cite[\S 9.2.1]{kousik2018bridging}; RTD attempts to find the best dynamically feasible trajectory to the waypoint.
RTD does not attempt to find the \emph{best waypoints} themselves (best, e.g., with respect to finding the shortest path to the global goal).
Such waypoints can be generated quickly by algorithms such as A* or RRT* by ignoring dynamic feasibility \cite{karaman_rrt_star,kousik2018bridging}.

\subsubsection{Polynomial Optimization Problems}
POPs require minimizing (or maximizing) a polynomial objective function, subject to polynomial equality or inequality constraints.
As a fundamental class of problems in non-convex optimization, POPs arise in various applications, including signal processing \cite{qi2003multivariate,thng1993derivative,mariere2003blind}, quantum mechanics \cite{dahl2007tensor,gurvits2003classical}, control theory \cite{kamyar2014polynomial,kousik2018bridging,ben2015_poly_dyn_BA}, and robotics \cite{rosen2019_se_sync,mangelson2019guaranteed}.
This paper presents a novel parallelized constrained Bernstein Algorithm for solving POPs.

The difficulty of solving a POP increases with the dimension of the cost and constraints, with the number of constraints, and with the number of optima \cite{nocedal2006numerical}.
Existing methods attempt to solve POPs while minimizing time and memory usage (i.e., complexity).
Doing so typically requires placing limitations on one of these axes of difficulty to make solving a POP tractable.
These methods broadly fall into the following categories: derivative-based, convex relaxation, and branch-and-bound.

Drivative-based methods use derivatives (and sometimes Hessians) of the cost and constraint functions, along with first- or second-order optimality conditions \cite[\S 12.3,~\S 12.5]{nocedal2006numerical}, to attempt to find optimal, feasible solutions to nonlinear problems such as POPs.
These methods can find local minima of POPs rapidly despite high dimension, a large number of constraints, and high degree cost and constraints \cite[Chapter 19.8]{nocedal2006numerical}.
However, these methods do not typically converge to global optima without requiring assumptions on the problem and constraint structure (e.g., \cite{qi2004global}).

Convex relaxation methods attempt to find global optima by approximating the original problem with a hierarchy of convex optimization problems.
These methods can be scaled to high-dimensional problems (up to 10 dimensions), at the expense of limits on the degree and sparse structure of the cost function; furthermore, they typically struggle to handle large numbers of constraints (e.g., the hundreds of constraints that arise in RTD's POPs), unless the problem has low-rank or sparse structure \cite{rosen2019_se_sync}.
Well-known examples include the lift-and-project linear program procedure \cite{balas1993lift}, reformulation-linearization technique \cite{sherali1990hierarchy}, and Semi-Definite Program (SDP) relaxations \cite{lasserre2001global,rosen2019_se_sync}.
By assuming structure such as homogeneity of the cost function or convexity of the domain and constraints, one can approximate solutions to a POP in polynomial-time, with convergence to global optima in the limit \cite{de2006ptas,ling2009biquadratic,luo2010semidefinite,so2011deterministic,he2010approximation}.
Convergence within a finite number of convex hierarchy relaxations is possible under certain assumptions (e.g., a limited number of equality constraints \cite{nie2013exact,lasserre2017bounded}).

Branch-and-bound methods perform an exhaustive search over the feasible region.
These methods are typically limited to up to four dimensions, but can handle large numbers of constraints and high degree cost and constraints.
Examples include interval analysis techniques \cite{hansen2003global,vaidyanathan1994global} and the Bernstein Algorithm (BA) \cite{garloff1993bernstein,nataraj2011constrained,sassi2015bernstein}.
Interval analysis requires cost and constraint function evaluations in each iteration, and therefore can be computationally slow.
BA, on the other hand, does not evaluate the cost and constraint functions; instead, BA represents the coefficients of the cost and constraints in the Bernstein basis, as opposed to the monomial basis.
The coefficients in the Bernstein basis provide lower and upper bounds on the polynomial cost and constraints over box-shaped subsets of Euclidean space by using a subdivision procedure \cite{garloff1985convergent,nataraj2007new}.
Note, one can use the Bernstein basis to transform a POP into a linear program (LP) on each subdivided portion of the problem domain, which allows one to find tighter solution bounds that given by the Bernstein coefficients alone \cite{sassi2015bernstein}.
Since subdivision can be parallelized \cite{dhabe2017parallel}, the time required to solve a POP can be greatly reduced by implementing BA on a Graphics Processing Unit (GPU).
However, a parallelized implementation or bounds on the rate of convergence of BA with constraints has not yet been shown in the literature.
Furthermore, to the best of our knowledge, BA has not been shown as a practical method for solving problems in real-time robotics applications.

\subsection{Contributions and Paper Organization}
In this paper, we make the following contributions.
First, we propose a \textbf{Parallel Constrained Bernstein Algorithm (PCBA)} (\S\ref{sec:constrained_BA}).
Second, we prove that PCBA always finds an \emph{optimal solution} (if one exists), and prove bounds on PCBA's time and memory usage (\S\ref{sec:complexity_analysis}).
Third, we evaluate PCBA on a suite of well-known POPs in comparison to the Bounded Sums-of-Squares (BSOS) \cite{lasserre2017bounded} solver and a generic nonlinear solver (MATLAB's \fmincon) (\S\ref{sec:benchmarking_PCBA}).
Fourth, we apply PCBA to RTD to make \textbf{RTD*}, a provably-safe, optimal, and real-time trajectory planning algorithm for mobile robots (\S\ref{sec:hardware_demo}), thereby demonstrating \defemph{dynamic feasibility} and \defemph{liveness}.
The remainder of the paper is organized as follows. 
\S\ref{sec:preliminaries} defines notation for RTD and POPs.
\S\ref{sec:conclusion} draws conclusions.
Appendix \ref{app:test_functions} lists benchmark problems for PCBA.
\section{Preliminaries}\label{sec:preliminaries}

This section introduces notation, RTD, POPs, the Bernstein form, and subdivision.

\subsection{Notation}
\subsubsection{Polynomial Notation}
We follow the notation in \cite{nataraj2011constrained}.
Let $x := (x_1, x_2, \cdots, x_l) \in \R^l$ be real variable of dimension $l \in \N$.
A multi-index $J$ is defined as $J := (j_1, j_2, \cdots, j_l) \in \N^l$
and the corresponding multi-power $x^J$ is defined as $x^J := (x_1^{j_1}, x_2^{j_2}, \cdots, x_l^{j_l}) \in \R^l$.
Given another multi-index $N := (n_1, n_2, \cdots, n_l) \in \N^l$ of the same dimension, an inequality $J \leq N$ should be understood component-wise.
An $l$-variate polynomial $p$ in canonical (monomial) form can be written as
\begin{equation}\label{eq:poly_monomial_basis}
    p(x)  = \sum_{J \leq N} a_J x^J, \quad x \in \R^l,
\end{equation}
with coefficients $a_J \in \R$ and some multi-index $N \in \N^l$.
The space of polynomials of degree $d \in \N$, with variable $x \in \R^l$, is $\R_d[x]$.

\begin{defn}\label{def:degree_and_multi_degree}
We call $N \in \N^l$ the \defemph{multi-degree} of a polynomial $p$; each $i$\ts{th} element of $N$ is the maximum degree of the variable $x_i$ out of all of the monomials of $p$.
We call $d \in \N$ the \defemph{degree} of $p$; $d$ is the maximum sum, over all monomials of $p$, of the powers of the variable $x$.
That is, $d = ||N||_1$, where $||\cdot||_1$ is the sum of the elements of a multi-index.
\end{defn}

\subsubsection{Point and Set Notation}
Let $\x := [\underline{x}_1, \overline{x}_1 ] \times \cdots \times [\underline{x}_l, \overline{x}_l] \subset \R^l$ denote a general $l$-dimensional box in $\R^n$, with $-\infty < \underline{x}_\mu < \overline{x}_\mu < +\infty$ for each $\mu = 1, \cdots, l$.
Let $\uu{} := [0,1]^l \subset \R^l$ be the $l$-dimensional unit box.
Denote by $|\x|$ the maximum width of a box $\x$, i.e. $|\x| = \max \{ \overline{x}_\mu - \underline{x}_\mu:\ \mu = 1,\cdots,l\}$.
For any point $y \in \R^l$, denote by $\|y\|$ the Euclidean norm of $y$,
and denote by $\B_R(y)$ the closed Euclidean ball centered at $y$ with radius $R>0$.

\subsubsection{RTD Notation}
Let $T = [0,\tfin]$ denote the time interval of a single plan (in a single receding-horizon iteration).
Let $X \subset \R^{n_X}$ denote the robot's state space, and $U \subset \R^{n_U}$ denote the robot's control inputs where $n_X, n_U \in \N$.
The robot is described by dynamics $f\hi: T \times X \times U \to \R^{n_X}$, which we call a \defemph{high-fidelity model}.
The parameterized trajectories are described by $f: T \times X \times Q \to \R^{n_X}$, where $Q \subset \R^{n_Q}$ is the space of trajectory parameters.
A point $q \in Q$ parameterizes a trajectory $x\des: T \to X$.
For any $q$, the robot uses a feedback controller $u_q: T\times X \to U$ to track the trajectory parameterized by $q$ (we say it tracks $q$ for short).

\subsection{Polynomial Optimization Problems}
We denote a POP as follows:
\begin{align}
\label{eq:POP}
    \begin{array}{cll}
         \underset{x \in D \subset \R^l}{\min} & p(x) & \\
         \regtext{s.t} & g_i(x) \leq 0 & i = 1,\cdots,\,\Nineq\\
                       & h_j(x) = 0 & j = 1,\cdots,\,\Neq.
    \end{array}\tag{P}
\end{align}
The decision variable is $x \in D \subset \R^n$, where $D$ is a compact, box-shaped domain and $l\in \N$ is the \defemph{dimension} of the program.
The cost function is $p \in \R_d[x]$, and the constraints are $g_i, h_j \in \R_d[x]$ ($\Nineq, \Neq \in \N$).
We assume for convenience that $d$ is the greatest degree amongst the cost and constraint polynomials; we call $d$ the \defemph{degree of the problem}.

\subsection{Reachability-based Trajectory Design}\label{sec:preliminaries:RTD}

Recall that RTD begins by computing an FRS offline for a robot tracking a parameterized set of trajectories.
At runtime, in each receding-horizon planning iteration, the FRS is used to construct a POP as in \eqref{eq:POP}; solving this POP is equivalent to generating a new trajectory plan.
We now briefly describe how this POP is constructed (see \cite{kousik2018bridging} for details).

We define the FRS $F \subset X\times Q$ as the set of all states reachable by the robot when tracking any parameterized trajectory:
\begin{align}\begin{split}
    F = \Big\{(x,q) \in X\times Q~|~&\exists\ t \in T\ \regtext{s.t.}\ x = \hat{x}(t),\ \hat{x}(0) \in X_0\ \regtext{and} \\
        &\dot{\hat{x}}(\tau) = f\hi(\tau,\hat{x}(\tau),u_q(\tau,\hat{x}(\tau)))\,\forall\,\tau \in T\Big\}
\end{split}\end{align}
where $X_0 \subset X$ is the set of valid initial conditions for the robot in any planning iteration.
To implement RTD, we conservatively approximate the FRS using sums-of-squares programming.
In particular, we compute a polynomial $w \in \R_d[q]$ ($d = 10$ or $12$) for which the 0-superlevel set contains the FRS:
\begin{align}\label{eq:w_FRS_polynomial}
    (x,q) \in F \implies w(x,q) \geq 0.
\end{align}
See \cite[Section 3.2]{kousik2018bridging} for details on how to compute $w$.

At runtime, we use $w$ to identify unsafe trajectory plans.
If $\{x_i\}_{i = 1}^n \subset X$ is a collection of points on obstacles, then we solve the following program in each planning iteration:
\begin{align}\label{prog:online_trajopt}
\begin{array}{cl}
     \underset{q\,\in\,Q}{\regtext{argmin}} & p(q) \\
     \regtext{s.t} & w(x_i,q) < 0~\forall~i = 1,\cdots,n,
\end{array}
\end{align}
where $p$ is a user-constructed polynomial cost function (see \eqref{prog:online_trajopt_POP} in \S\ref{sec:hardware_demo} as an example).
Note that, by \eqref{eq:w_FRS_polynomial}, the set of safe trajectory plans is open, so we implement the constraints in \eqref{prog:online_trajopt} as $w(x_i,q) + \epsilon_q \leq 0$, $\epsilon_q \approx 10^{-4}$.
Critically, any feasible solution to \eqref{prog:online_trajopt} is provably dynamically feasible (and collision-free) \cite[Theorem 68]{kousik2018bridging}.
To understand RTD's online planning in more detail, see \cite[Algorithm 2]{kousik2018bridging}.

In this work, instead of using a derivative-based method to solve \eqref{prog:online_trajopt}, we use our proposed PCBA method, which takes advantage of the polynomial structure of $p$ and $w$.
Next, we discuss Bernstein polynomials.

\subsection{Bernstein Form}
A polynomial $p$ in monomial form \eqref{eq:poly_monomial_basis} can be expanded into \emph{Bernstein form} over an arbitrary $l$-dimensional box $\x$ as
\begin{equation}
\label{eq:BF}
    p(x) = \sum_{J \leq N} B_J^{(N)}(\x) \, b_J^{(N)} (\x, x),
\end{equation}
where $b_J^{(N)}(\x, \cdot)$ is the $J$\ts{th} multivariate \emph{Bernstein polynomial} of multi-degree $N$ over $\x$,
and $B_J^{(N)}(\x)$ are the corresponding \emph{Bernstein coefficients} of $p$ over $\x$.
A detailed definition of Bernstein form is available in \cite{hamadneh2018bounding}.
Note that the Bernstein form of a polynomial can be determined quickly  \cite{titi2017fast}, by using a matrix multiplication on a polynomial's monomial coefficients, with the matrix determined by the polynomial degree and dimension.
This matrix can be precomputed, and the conversion from monomial to Bernstein basis only needs to happen once for the proposed method (see Algorithm \ref{alg:pcba} in \S\ref{sec:constrained_BA}).

For convenience, we collect all such Bernstein coefficients in a multi-dimensional array $B(\x) := \big(B_J^{(N)}(\x)\big)_{J \leq N}$, which is called a \emph{patch}.
We denote by $\min B(\x)$ (resp. $\max B(\x)$) the minimum (resp. maximum) element in the patch $B(\x)$.
The range of polynomial $p$ over $\x$ is contained within the interval spanned by the extrema of $B(\x)$, formally stated as the following theorem:

\begin{thm}[{\cite[Lemma~2.2]{nataraj2011constrained}}]
\label{thm:enclosure}
Let $p$ be a polynomial defined as in \eqref{eq:BF} over a box $\x$. Then, the following property holds for a patch $B(\x)$ of Bernstein coefficients
\begin{equation}
    \min B(\x) \leq p(x) \leq \max B(\x), \quad \forall x \in \x.
\end{equation}
\end{thm}

\noindent This theorem provides a means to obtain enclosure bounds of a multivariate polynomial over a box by transforming the polynomial to Bernstein form.
This range enclosure can be further improved either by degree elevation or by subdivision.
This work uses subdivision, discussed next, to refine the bounds.

\subsection{Subdivision Procedure}
\label{subsec:subdivision}
Consider an arbitrary box $\x \subset \R^l$.
The range enclosure in Theorem \ref{thm:enclosure} is improved by subdividing $\x$ into subboxes and computing the Bernstein patches over these subboxes.
A \defemph{subdivision} in the $r$\ts{th} direction ($1 \leq r \leq l$) is a bisection of $\x$ perpendicular to this direction.
That is, let
\begin{equation}
    \x := [\underline{x}_1, \overline{x}_1] \times \cdots \times [\underline{x}_r, \overline{x}_r] \times \cdots \times [\underline{x}_l, \overline{x}_l]
\end{equation}
be an arbitrary box over which the Bernstein patch $B(\x)$ is already computed.
By subdividing $\x$ in the $r$\ts{th} direction we obtain two subboxes $\x_L$ and $\x_R$, defined as
\begin{equation}
    \begin{split}
        \x_L &= [\underline{x}_1, \overline{x}_1] \times \cdots \times [\underline{x}_r, (\underline{x}_r + \overline{x}_r)/2] \times \cdots \times [\underline{x}_l, \overline{x}_l], \\
        \x_R &= [\underline{x}_1, \overline{x}_1] \times \cdots \times [(\underline{x}_r + \overline{x}_r)/2, \overline{x}_r] \times \cdots \times [\underline{x}_l, \overline{x}_l].
    \end{split}
\end{equation}
Note that we have subdivided $\x$ by halving its width in the $r$\ts{th} direction; we choose $1/2$ as the subdivision parameter in this work, but one can choose a different value in $[0,1]$ (see \cite[Equation (10)]{nataraj2011constrained}).

The new Bernstein patches, $B(\x_L)$ and $B(\x_R)$, can be computed by a finite number of linear transformations \cite[\S2.2]{nataraj2011constrained}:
\begin{equation}
\begin{split}
    B(\x_L) &= M_{r,L} B(\x), \\
    B(\x_R) &= M_{r,R} B(\x).
\end{split}
\end{equation}
where $M_{r,L}$ and $M_{r,R}$ are constant matrices, which can be precomputed, for each $r$ (notice that \cite[Equation (10)]{nataraj2011constrained} obtains $B(\x_L)$ with linear operations on $B(\x)$).
The patches and one iteration of the subdivision procedure are shown in Figure \ref{fig:ba_iteration_example}.

\begin{rem}
To reduce wordiness, we say that we \defemph{subdivide a patch} to mean the subdivision of a single subbox into two subboxes \emph{and} the computation of the corresponding Bernstein patches for the POP cost and constraints.
\end{rem}

By repeatedly applying the subdivision procedure and Theorem \ref{thm:enclosure}, the bounds on the range of polynomial in a subbox can be improved.
In fact, such bounds can be exact in the limiting sense if the subdivision is applied evenly in all directions:
\begin{thm}[{\cite[Theorem~2]{malan1992b}}]
\label{thm:err_bound}
Let $\x^{(n)}$ be a box of maximum width $2^{-n}$ ($n \in \N$) and let $B(\x^{(n)})$ be the corresponding Bernstein patch of a given polynomial $p$, then
\begin{equation}
    \min B(\x^{(n)}) \leq \min_{x \in \x^{(n)}} p(x) \leq \min B(\x^{(n)}) + \zeta \cdot 2^{-2n}
\end{equation}
where $\zeta$ is a non-negative constant that can be given explicitly independent of $n$.
\end{thm}

Notice from the proof of Theorem \ref{thm:err_bound} that changing the sign of $p$ does not change the value of $\zeta$.
By substituting $p$ with $-p$ in Theorem \ref{thm:err_bound}, one can easily show a similar result holds for the maximum of $p$ over $\x^{(n)}$:
\begin{cor}
\label{cor:err_bound}
    Let $\x^{(n)}$ and $B(\x^{(n)})$ be as in Theorem \ref{thm:err_bound}.
    Then,
    \begin{equation}
        \max B(\x^{(n)}) - \zeta \cdot 2^{-2n} \leq \max_{x \in \x^{(n)}} p(x) \leq \max B(\x^{(n)}),
    \end{equation}
    where $\zeta$ is the same non-negative constant as in Theorem \ref{thm:err_bound}.
\end{cor}

Theorems \ref{thm:err_bound} and Corollary \ref{cor:err_bound} provide shrinking bounds for values of a polynomial over subboxes as the subdivision process continues.
By comparing the bounds over all subboxes, one can argue that the minimizers of a polynomial may appear in only a subset of the subboxes.
This idea underlies the \emph{Bernstein Algorithm} for solving POPs \cite{nataraj2007new,nataraj2011constrained}, discussed next.
\section{Parallel Constrained Bernstein Algorithm}\label{sec:constrained_BA}

\begin{figure}[t]
\centering
\includegraphics[width=\columnwidth]{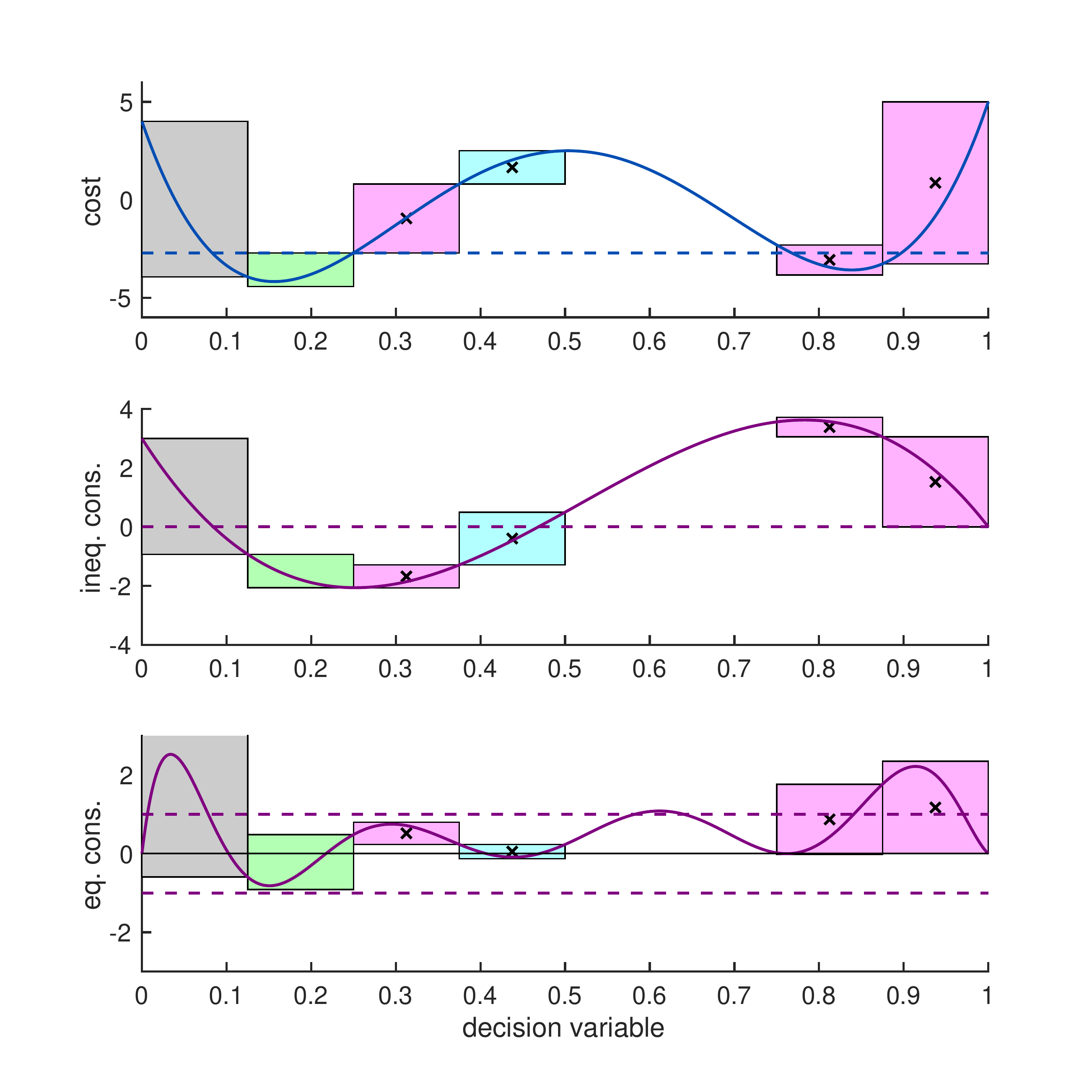}
\caption{The 3\ts{rd} iteration of PCBA (Algorithm \ref{alg:pcba}) on a one-dimensional polynomial cost (top) with one inequality constraint (middle) and one equality constraint (bottom).
The rectangles represent Bernstein patches (as in \S\ref{subsec:subdivision}), where the horizontal extent of each patch corresponds to an interval of the decision variable over which Bernstein coefficients are computed.
The top and bottom of each patch represent the maximum and minimum Bernstein coefficients, which bound the cost and constraint polynomials on the corresponding interval.
As per Definition \ref{def:feasible}, the green patch is feasible, the pink patches are infeasible, and the grey patches are undecided; the purple dashed lines show the inequality constraint cut-off (zero) and the equality constraint tolerance $\eeq = 1$ (note that $\eeq$ is chosen to be this large only for illustration purposes).
Per Definition \ref{def:suboptimal}, the light blue patch is suboptimal; the blue dashed line in the top plot is the current solution estimate (Definition \ref{def:current_best_soln}).
The infeasible and suboptimal patches are each marked with $\times$ for elimination (Algorithm \ref{alg:elim}), since they cannot contain the global optimum (Theorem \ref{thm:cut-off}); the feasible and undecided patches are kept for the next iteration.
}
\label{fig:ba_iteration_example}
\end{figure}

This section proposes the Parallel Constrained Bernstein Algorithm (PCBA, Algorithm \ref{alg:pcba}) for solving a general POP.
We extend the approach in \cite{nataraj2011constrained}.
This approach utilizes Bernstein form to obtain upper and lower bounds of both objective and constraint polynomials (Theorem \ref{thm:enclosure}), iteratively improves such bounds using subdivision (Theorem \ref{thm:err_bound} and Corollary \ref{cor:err_bound}), and removes patches that are cannot contain a solution (Theorem \ref{thm:cut-off}).
We discuss the algorithm, the list used to store patches, tolerances and stopping criteria, subdivision, a cut-off test for eliminating patches, and the advantages and disadvantages of PCBA.
The next section, \S\ref{sec:complexity_analysis}, proves PCBA finds globally optimal solutions to POPs, up to user-specified tolerances.

\begin{algorithm}[H]
\textbf{Inputs:}
    Polynomials $p$, $\{g_i\}_{i=1}^\Nineq$, $\{h_j\}_{j=1}^\Neq$ as in \eqref{eq:POP}, of dimension $l$; optimality tolerance $\epsilon > 0$, step tolerance $\delta > 0$, and equality constraint tolerance $\eeq > 0$; and maximum number of patches $M \in \N$ and of iterations $N \in \N$.

\textbf{Outputs:}
    Estimate $p^* \in \R$ of optimal solution, and subbox $\x^* \subset \uu$ containing optimal solution.

\textbf{Algorithm:}
\begin{algorithmic}[1]
\State Initialize patches of $p$, $g_i$, and $h_j$ over $l$-dimensional initial domain box $\uu$ as in \cite{titi2017fast} \label{alg:pcba:init:patches}

    $[B_p(\uu), B_{gi}(\uu), B_{hj}(\uu)] \gets$ InitPatches$\left(p, g_i, h_j\right)$
        
\State Initialize lists of undecided patches and patch extrema on the GPU \label{alg:pcba:init:lists}
    
        $\LL \gets \left\{\left(\uu, B_p(\uu), B_{gi}(\uu), B_{hj}(\uu)\right)\right\}$
        
        $\LL_\text{bounds} \gets \{\}$
        
\State Initialize iteration count and subdivision direction \label{alg:pcba:init:counters}
        
        $n \gets 1$, $r \gets 1$

\State Test for sufficient memory (iteration begins here) \label{alg:pcba:mem_test}
    
    \textbf{if} $2\times\regtext{length}(\LL)>M$ \textbf{then} go to \ref{alg:pcba:return}
    
    \textbf{else} continue
    
    \textbf{end if}
    
\State \blue{(Parallel)} Subdivide each patch in $\LL$ in the $r$\ts{th} direction to create two new patches using Algorithm \ref{alg:subdivision} \label{alg:pcba:subdiv} 
    
    $\LL \gets$ Subdivide($\LL,r$)
    
\State \blue{(Parallel)} Find bounds of $p$, $g_i$, and $h_j$ on each new patch using Algorithm \ref{alg:find_bounds} \label{alg:pcba:find_bounds}
         
    $\LLbounds \gets$ FindBounds$(\LL)$
           
\State  Estimate upper bound $\pu$ of the global optimum as the least upper bound of all feasible patches, and determine which patches to eliminate using Algorithm \ref{alg:cut-off} \label{alg:pcba:cut-off}
           
    $[\pu,\pl,\LLsave,\LLelim] \gets$ CutOffTest$(\LL_\text{bounds})$;
           
\State Test if problem is feasible \label{alg:pcba:check_infeasible}

    \textbf{if} $\regtext{length}(\LLsave)=0$ \textbf{then} go to \ref{alg:pcba:return_infeasible}
    
\State Test stopping criteria for all $\left(\x,B_p(\x),B_{gi}(\x),B_{hj}(\x)\right) \in \LL$ \label{alg:pcba:stopping_criteria}

    \textbf{if} $\pu - \pl \leq \epsilon$
    
    \hspace*{0.25cm}\textbf{and} $|\x| \leq\delta$
    
    \hspace*{0.25cm}\textbf{and} $-\epsilon_{eq}\leq\min B_{hj}(\x)\leq\max B_{hj}(\x)\leq\epsilon_{eq}$
    
    \textbf{then} go to \ref{alg:pcba:return}
    
    \textbf{end if}
    
\State \blue{(Parallel)} Eliminate infeasible, suboptimal patches using Algorithm \ref{alg:elim} \label{alg:pcba:eliminate}

    $\LL \gets$ Eliminate$(\LL,\LLsave,\LLelim)$;
    
\State Prepare for next iteration \label{alg:pcba:prep_next_iter}
    
    $r\gets(\regtext{mod}(r+1,l))+1$
    
    \textbf{if} $r = 1$ \textbf{then} $n\gets n + 1$
    
    \textbf{end if}
    
    \textbf{if} $n = N$ \textbf{then} go to Step \ref{alg:pcba:return} 
    
    \textbf{else} go to Step \ref{alg:pcba:mem_test}
    
    \textbf{end if}
    
\State Return current best approximate solution  \label{alg:pcba:return}

    $p^* \gets \pu$
    
    $\x^* \gets \x$ for which $\max B_p(\x) = \pu$
    
    \Return $p^*$, $\x^*$
    
\State No solution found (problem infeasible) \label{alg:pcba:return_infeasible}
\end{algorithmic}
    \caption{Parallel Constrained Bernstein Algorithm (PCBA)}
    \label{alg:pcba}
\end{algorithm}

Before proceeding, we make an assumption for notational convenience, and to make the initial computation of Bernstein patches easier.
\begin{assum}
Without loss of generality, the domain of the decision variable is the unit box (i.e., $D = \uu$), since any nonempty box in $\R^l$ can be mapped affinely onto $\uu$ \cite{titi2017fast}.
\end{assum}

\subsection{Algorithm Summary}

We now summarize PCBA, implemented in Algorithm \ref{alg:pcba}.
The algorithm is initialized by computing the Bernstein patches of the cost and constraints on the domain $\uu$ (Line \ref{alg:pcba:init:patches}).
Subsequently, PCBA subdivides each patch as in Remark \ref{rem:subdiv} (Line \ref{alg:pcba:subdiv} and Algorithm \ref{alg:subdivision}).
Then, PCBA finds the upper and lower bounds of each new patch (Line \ref{alg:pcba:find_bounds} and Algorithm \ref{alg:find_bounds}).
These bounds are used to determine which patches are feasible, infeasible, and undecided as in Definition \ref{def:feasible} (Line \ref{alg:pcba:cut-off}; see Algorithm \ref{alg:cut-off} and Theorem \ref{thm:cut-off}).
Algorithm \ref{alg:cut-off} also determines the current solution estimate (the smallest upper bound over all feasible patches), and marks any patches that are suboptimal as in Definition \ref{def:suboptimal}.
If every patch is infeasible (Line \ref{alg:pcba:check_infeasible}), PCBA returns that the problem is infeasible (Line \ref{alg:pcba:return_infeasible}); otherwise, PCBA checks if the current solution estimate meets user-specified tolerances (Line \ref{alg:pcba:stopping_criteria}).
If the tolerances are met, PCBA returns the solution estimate (Line \ref{alg:pcba:return}).
Otherwise, PCBA eliminates all infeasible and suboptimal patches (Line \ref{alg:pcba:eliminate} and Algorithm \ref{alg:elim}), then moves to the next iteration (Line \ref{alg:pcba:prep_next_iter}).
Note, algorithms \ref{alg:subdivision}, \ref{alg:find_bounds}, and \ref{alg:elim} are parallelized.

\subsection{Items and The List}\label{subsec:list}
Denote an \defemph{item} as the tuple $\ell = (\x, B_p(\x), B_{gi}(\x)), B_{hj}(\x))$, where $B_{gi}(\x)$ (resp. $B_{hj}(\x)$) is shorthand for the set of patches $\{B_{g_i}(\x)\}_{i = 1}^{\Nineq}$ (resp. $\{B_{h_j}(\x)\}_{j = 1}^{\Neq}$).
We use the following notation for items.
If $\ell = \itmx$, then $\ell_1 = \x$, $\ell_2 = B_p(\x)$, $\ell_3 = B_{gi}(\x)$, and $\ell_4 = B_{hj}(\x)$.

We denote the \defemph{list} $\LL = \{\ell_\mu: \mu = 1,\cdots,N_{\LL}\}$, $N_{\LL} \in \N$, indexed by $\mu \in \N$.
PCBA adds and removes items from $\LL$ by assessing the feasibility and optimality of each item.

\subsection{Tolerances and Stopping Criteria}\label{subsec:tolerances}

Recall that, by Theorem \ref{thm:enclosure}, Bernstein patches provide upper and lower bounds for polynomials over a box.
From Theorem \ref{thm:err_bound} and Corollary \ref{cor:err_bound}, as we subdivide $\uu$ into smaller subboxes, the bounds of the Bernstein patches on each subbox more closely approximate the actual bounds of the polynomial.
While the bounds will converge to the value of the polynomial in the limit as the maximum width of subboxes goes to zero, to ensure the algorithm terminates, we must set tolerances on optimality and equality constraint satisfaction (the equality constraints $h_j(x) = 0$ may not be satisfied for \emph{all} points in certain boxes).
During optimization one is also usually interested in finding the optimizer of the cost function up to some resolution. 
In our case, this resolution corresponds to the maximum allowable subbox width which we refer to as the \defemph{step tolerance}.

\begin{defn}\label{def:tolerances}
We denote the \defemph{optimality tolerance} as $\epsilon > 0$, the \defemph{equality constraint tolerance} as $\eeq > 0$, and the \defemph{step tolerance} as $\delta > 0$.
We terminate Algorithm \ref{alg:pcba} either when $\LL$ is empty (the problem is infeasible) or when there exists an item $(\x, B_p(\x), B_{gi}(\x)), B_{hj}(\x)) \in \LL$ that satisfies all of the following conditions:
\begin{enumerate}[label=(\alph*),ref=(\alph*)]
    \item $|\x| \leq \delta$,
    \item $\max B_{gi}(\x) \leq 0$ for all $i = 1, \cdots, \Nineq$,
    \item $-\eeq \leq \min B_{hj}(\x) \leq 0 \leq \max B_{hj}(\x) \leq \eeq$ for all $j = 1, \cdots, \Neq$, and
    \item $\max B_p(\x) - \min B_p(\y) \leq \epsilon$ for all $\y \in \LL$.
\end{enumerate}
\end{defn}
\noindent We discuss feasibility in more detail in section \ref{subsec:cut-off_test}
Note that, to implement the step tolerance $\delta$, since we subdivide by halving the width of each subbox, we need only ensure that sufficiently many iterations have passed.

Note that we do not set a tolerance on inequality constraints, since these are ``one-sided'' constraints; for any inequality constraint $g_i$ and subbox $\x$, we satisfy the constraint if $\max B_{gi}(\x) \leq 0$ (see Definition \ref{def:feasible} and Theorem \ref{thm:roc:constrained}).

\subsection{Subdivision}

Recall that subdivision is presented in \S\ref{subsec:subdivision}.
We implement subdivision with Algorithm \ref{alg:subdivision}.
Since the subdivision of one Bernstein patch is computationally independent of another, each subdivision task is assigned to an individual GPU thread, making Algorithm \ref{alg:subdivision} parallel.

Note that the subdivision of Bernstein patches can be done in any direction, leading to the question of how to select the direction in practice.
Example rules are available in the literature, such as maximum width \cite[\S 3]{ratz1995selection}, derivative-based \cite[\S 3]{zettler1998robustness}, or a combination of the two \cite[\S 3]{ratz1995selection}.
In the context of constrained optimization, the maximum width rule is usually favored over derivative-based rules for two reasons:
first, computing the partial derivatives of all constraint polynomials can introduce significant computational burden, especially when the number of constraints is large (see \S\ref{subsec:example:constraints});
second, the precision of Bernstein patches as bounds to the polynomials depends on the \emph{maximum} width of each subbox (Theorem \ref{thm:err_bound} and Corollary \ref{cor:err_bound}), so it is beneficial to subdivide along the direction of maximum width for better convergence results.

In each $n$\ts{th} iteration of PCBA, we subdivide in each direction $r$, in the order $1,2,\cdots,l$.
We halve the width of each subbox each time we subdivide, leading to the following remark.
\begin{rem}
\label{rem:subdiv}
    In the $n$\ts{th} iteration, the maximum width of any subbox in $\LL$ is $2^{-n}$.
\end{rem}

\begin{algorithm}[ht]
\begin{algorithmic}[1]
\State $K \gets \regtext{length}(\LL)$
\ParFor{$k\in\{1,\ldots,K\}$}
    \State $\left(\x, B_p(\x), B_{gi}(\x), B_{hj}(\x)\right)\gets\LL[k]$;
    
    \State Subdivide $\x$ along the $r$\ts{th} direction into $\x_L$ and $\x_R$
    
    \State Compute patches $B_p(\x_L)$ and $B_p(\x_R)$
    
    \State Compute patches $B_{gi}(\x_L)$ and $B_{gi}(\x_R)$
    
    \State Compute patches $B_{hj}(\x_L)$ and $B_{hj}(\x_R)$
    
    \State $\LL[k] \gets (\x_L, B_p(\x_L), B_{gi}(\x_L), B_{hj}(\x_L))$
    
    \State $\LL[k + K] \gets (\x_R, B_p(\x_R), B_{gi}(\x_R), B_{hj}(\x_R))$
\EndParFor
\State \Return $\LL$
\end{algorithmic}
    \caption{$\LL =$ Subdivision$(\LL,r)$ \blue{(Parallel)}}
    \label{alg:subdivision}
\end{algorithm}

\subsection{Cut-Off Test}\label{subsec:cut-off_test}
Subdivision would normally occur for every patch in every iteration, leading to exponential memory usage ($2^n$ patches in iteration $n$).
However, by using a cut-off test, some patches can be deleted, reducing both the time and memory usage of PCBA (see \S\ref{sec:complexity_analysis} for complexity analysis).
To decide which patches are to be eliminated, we require the following definitions.

\begin{defn}\label{def:feasible}
An item $\itmx \in \LL$ is \defemph{feasible} if both of the following hold:
\begin{enumerate}[label=(\alph*),ref=(\alph*)]
    \item $\max B_{gi}(\x) \leq 0$ for all $i = 1, \cdots, \Nineq$, and
    \item $-\eeq \leq \min B_{hj}(\x) \leq 0 \leq \max B_{hj}(\x) \leq \eeq$ for all $j = 1,\cdots,\Neq$.
\end{enumerate}
An item is \defemph{infeasible} if any of the following hold:
\begin{enumerate}[label=(\alph*),ref=(\alph*)]
\setcounter{enumi}{2}
    \item $\min B_{gi}(\x) > 0$ for at least one $i = 1, \cdots, \Nineq$, or
    \item $\min B_{hj}(\x) > 0$ for at least one $j = 1,\cdots,\Neq$, or
    \item $\max B_{hj}(\x) < 0$ for at least one $j = 1,\cdots,\Neq$.
\end{enumerate}
An item is \defemph{undecided} if it is neither feasible nor infeasible.
\end{defn}

\noindent Notice in particular, a feasible item must not be infeasible.

\begin{defn}\label{def:current_best_soln}
The \defemph{solution estimate} $\pu$ is the smallest upper bound of the cost over all feasible items in $\LL$:
\begin{align}
    \pu = \min\Big\{ \max \{\ell_2~|~\ell \in \LL,~\ell~\regtext{feasible}\}\Big\},
\end{align}
where $\ell_2 = B_p(\x)$ if $\ell = \itmx$.
\end{defn}

\begin{defn}\label{def:suboptimal}
An item $\itmx \in \LL$ is \defemph{suboptimal} if
\begin{align}
    \min B_p(\x) > \pu.
\end{align}
\end{defn}

\noindent
Note that Definitions \ref{def:current_best_soln} and \ref{def:suboptimal} are dependent on $\LL$; that is, for the purposes of PCBA, optimality is defined in terms of the elements of $\LL$.
We show in Corollary \ref{cor:cut-off} below how this notion of optimality coincides with optimality of the POP itself.

Feasible, infeasible, undecided, and suboptimal patches are illustrated in Figure \ref{fig:ba_iteration_example}.
Any item that is infeasible or suboptimal can be eliminated from $\LL$, because the corresponding subboxes cannot contain the solution to the POP (formalized in the following Theorem).
We call checking for infeasible and suboptimal items the \defemph{cut-off test}.

\begin{thm}[Cut-Off Test]
\label{thm:cut-off}
Let $\itmx \in \LL$ be an item.
If the item is infeasible (as in Definition \ref{def:feasible}) or suboptimal (as in Definition \ref{def:suboptimal}), then $\x$ does not contain a global minimizer of \eqref{eq:POP}.
Such item can be removed from the list $\LL$.
\end{thm}
\begin{proof}
Let $\itmx$ be an item in $\LL$. We only need to show:
\begin{enumerate}[label=(\alph*),ref=(\alph*)]
    \item if $\itmx$ is feasible, then all points in $\x$ are feasible (up to the tolerance $\eeq$), and \label{cut-off:cond:1}
    \item if $\itmx$ is infeasible, then all points in $\x$ are infeasible (up to the tolerance $\eeq$), and \label{cut-off:cond:2}
    \item if $\itmx$ is suboptimal, then all points in $\x$ are not optimal. \label{cut-off:cond:3}
\end{enumerate}
Note that \ref{cut-off:cond:1} and \ref{cut-off:cond:2} follows directly from Theorem \ref{thm:enclosure}.
To prove \ref{cut-off:cond:3}, let $\y \subset \uu$ be a subbox on which the solution estimate $\pu$ is achieved, that is, $\itmy$ is feasible and
\begin{align}
    \max B_p(\y) = \pu.
\end{align}
Let $y \in \y$ be arbitrary, then it follows from Theorem \ref{thm:enclosure} and the definition of suboptimality that
\begin{align}
    p(x) \geq \min B_p(\x) > \max B_p(\y) \geq p(y)
\end{align}
for all $x \in \x$. Since such point $y$ is necessarily feasible (obtained from condition \ref{cut-off:cond:2}), $x$ cannot be global minimum to the POP.
\end{proof}

\begin{cor}
\label{cor:cut-off}
Suppose there exists a (feasible) global minimizer $x^*$ of the POP \eqref{eq:POP}.
Then, while executing Algorithm \ref{alg:pcba}, there always exists an item $\itmx \in \LL$ such that $x^* \in \x$.
\end{cor}

\begin{proof}
This result is the contrapositive of Theorem \ref{thm:cut-off}.
\end{proof}

\begin{algorithm}[t]
\begin{algorithmic}[1]
\State $K \gets \regtext{length}(\LL)$
\ParFor{$k \in \{1,\ldots,K\}$}
    \State $\left(\x, B_p(\x), B_{gi}(\x), B_{hj}(\x)\right)\gets\LL[k]$
    
    \State Find $\min B_p(\x)$ and $\max B_p(\x)$ by parallel reduction
    
    \State Find $\min B_{gi}(\x)$ and $\max B_{gi}(\x)$ similarly
    
    \State Find $\min B_{hj}(\x)$ and $\max B_{hj}(\x)$ similarly
    
    \State $\LL_\text{bounds}[k] \gets 
        \left(\x,\begin{array}{ll}
        \{\min B_p(\x),&\max B_p(\x)\}  \\
        \{\min B_{gi}(\x),&\max B_{gi}(\x)\}  \\
        \{\min B_{hj}(\x),&\max B_{hj}(\x)\}  
        \end{array}\right)$
\EndParFor
\State \Return $\LL_\text{bounds}$
\end{algorithmic}
    \caption{$\LL_\text{bounds} =$ FindBounds$(\LL)$ \blue{(Parallel)}}
    \label{alg:find_bounds}
\end{algorithm}

\begin{algorithm}[t]
\begin{algorithmic}[1]
    \State $\pu\gets+\infty,\ \pl\gets+\infty$
    
    \State $K \gets \regtext{length}(\LLbounds)$
    
    \For{$k\in\{1,\ldots,K\}$}
    
        \State $\left(\x,\begin{array}{ll}
        \{\min B_p(\x),&\max B_p(\x)\}  \\
        \{\min B_{gi}(\x),&\max B_{gi}(\x)\}  \\
        \{\min B_{hj}(\x),&\max B_{hj}(\x)\}  
        \end{array}\right) \gets \LLbounds[k]$
        
        \If {$-\eeq \leq \min B_{hj}(\x) \leq 0 \leq \max B_{hj}(\x) \leq \eeq$}
        
            \If{$\max B_{gi}(\x) \leq0$}
            
                \State $\pu\gets\min(\pu,\max B_p(\x))$
            \EndIf
            
            \If{$\min B_{gi}(\x) \leq0$}
            
                \State $\pl\gets\min(\pl,\min B_p(\x))$
                
            \EndIf
        \EndIf
    \EndFor
    
    \State Initialize lists for indices of patches to save or eliminate
    $\LLsave \gets \{\}$, $\LLelim \gets \{\}$
    
    \For{$k\in\{1,\cdots,K\}$}
        \If{$-\eeq \leq \min B_{hj}(\x) \leq 0 \leq \max B_{hj}(\x) \leq \eeq$
        
            \textbf{and} $\min B_{gi}(\x) \leq 0$
            
            \textbf{and} $\min B_p(\x) \leq \pu$}
            \State Append $k$ to $\LLsave$
        \Else
            \State Append $k$ to $\LLelim$
        \EndIf
    \EndFor
    \State \Return $\pu,\pl,\LLsave$, $\LLelim$
\end{algorithmic}
    \caption{$[\pu,\pl,\LLsave,\LLelim]$ = CutOffTest$(\LL_\text{bounds})$}
    \label{alg:cut-off}
\end{algorithm}

\begin{algorithm}[ht]
\begin{algorithmic}[1]
    \State $K_{\regtext{save}} \gets \regtext{length}(\LL_\regtext{save})$
    
    \State $K_{\regtext{elim}} \gets \regtext{length}(\LL_\regtext{elim})$
    
    \State $K_{\regtext{replace}} \gets K_{\regtext{elim}} - 1$
    
    \If{$K_{\regtext{elim}} = 0$ \textbf{or} $\LL_\regtext{elim}[1] > K_{\regtext{elim}}$}
        \State \Return $\LL$
    \EndIf
    \For{$k \in \{1,\cdots,K_{\regtext{elim}}\}$}
        \If{$\LL_\regtext{elim}[k]\geq K_{\regtext{save}}$}
            \State $K_{\regtext{replace}}\gets k-1$
            \State \textbf{break}
        \EndIf
    \EndFor
    \ParFor{$k \in \{1,\cdots,K_{\regtext{replace}}\}$}
	   \State  $\LL[\LLelim[k]] \gets \LL[\LLsave[K_{\regtext{save}} + 1 - k]]$
    \EndParFor
    \State \Return $\LL$
\end{algorithmic}
    \caption{$\LL =$ Eliminate$(\LL,\LL_\regtext{save},\LL_\regtext{elim})$ \blue{(Parallel)}}
    \label{alg:elim}
\end{algorithm}

We implement the cut-off tests as follows.
Algorithm \ref{alg:find_bounds} (FindBounds) computes the maximum and minimum element of each Bernstein patch; Algorithm \ref{alg:cut-off} (CutOffTest) implements the cut-off tests and marks all subboxes to be eliminated with a list $\LLelim$; and Algorithm \ref{alg:elim} (Eliminate) eliminates the marked subboxes from the list $\LL$.
Algorithms \ref{alg:find_bounds} and \ref{alg:elim} are parallelizable, whereas Algorithm \ref{alg:cut-off} must be computed serially.

\subsection{Advantages and Disadvantages of PCBA}

PCBA has several advantages.
First, it always finds a global optimum (if one exists), subject to tolerances.
PCBA does not require an initial guess, and does not converge to local minima, unlike generic nonlinear solvers (e.g., \fmincon\  \cite{MatlabOTB}).
It also does not require tuning hyperparameters.
As we show in \S\ref{sec:complexity_analysis}, PCBA has bounded time and memory complexity under certain assumptions.
Finally, due to parallelization, PCBA is fast enought to enable RTD* for real-time, safe, optimal trajectory planning, which we demonstrate in \S\ref{sec:hardware_demo}.

However, PCBA also has several limitations in comparison to traditional approaches to solving POPs.
First, to prove the bounds on time and memory usage, at any global minimum, we require that active constraints are linearly independent, and that the Hessian of the cost function is positive definite (see Theorems \ref{thm:roc:unconstrained} and \ref{thm:roc:constrained}).
Furthermore, due to the number of Bernstein patches growing exponentially with the decision variable dimension, we have not yet applied PCBA to problems larger than four-dimensional.
\section{Complexity Analysis}\label{sec:complexity_analysis}

In this section, we prove that Algorithm \ref{alg:pcba} terminates by bounding the number of iterations of PCBA for both unconstrained and constrained POPs.
We also prove the number of Bernstein patches (i.e., the length of the list $\LL$ in Algorithm \ref{alg:pcba}) is bounded after sufficiently many iterations, under certain assumptions.
For convenience, in the remainder of this section we use $\x \in \LL$ as a shorthand notation for $\x = \ell_1$ where $\ell = \itmx \in \LL$.
The proofs of Theorems \ref{thm:uncon:memory}, \ref{thm:roc:constrained}, and \ref{thm:memory_usage:constrained} are in the Appendix.

\subsection{Unconstrained Case}\label{sec:complexity_analysis:unconstrained}
We first consider unconstrained POPs whose optimal solutions are not on the boundary of $\uu$.
Note, we can treat optimal solutions on the boundary of $\uu$ as having active linear constraints; see \S\ref{sec:complexity_analysis:constrained} for the corresponding complexity analysis.
In the unconstrained case, all points in $\uu$ are feasible, and we are interested in solving 
\begin{equation}
\label{eq:POP_unconstrained}
\min_{x\,\in\,\uu\,\subset\,\R} \quad p(x)
\end{equation}
where $p$ is an $l$-dimensional multivariate polynomial.
Given an optimality tolerance $\epsilon$ and step tolerance $\delta$, we bound the number of iterations to solve \eqref{eq:POP_unconstrained} with PCBA as follows:

\begin{thm}
\label{thm:roc:unconstrained}
Let $p$ in \eqref{eq:POP_unconstrained} be a multivariate polynomial of dimension $l$ with Lipschitz constant $L_p$.
Then the maximum number of iterations needed to solve \eqref{eq:POP_unconstrained} up to accuracy $\epsilon$ and $\delta$ is
\begin{equation}
\label{eq:uncon:K}
    N =\left\lceil \max\left\{
    -\log_2 \delta,
    -\frac{1}{2}\log_2 \left( \frac{\epsilon}{4 \zeta_p} \right), -\log_2 \left( \frac{\epsilon}{2 L_p \sqrt{l}} \right) \right\} \right\rceil,
\end{equation}
where $\zeta_p$ is the constant in Theorem \ref{thm:err_bound} corresponding to polynomial $p$, and $\ceil{\cdot}$ rounds up to the nearest integer.
\end{thm}

\begin{proof}
Let $n$ be the current iteration number.
It is sufficient to show that, for all $n > N$, there exists a subbox $\x \in \LL$ of $\uu$ such that the following conditions hold:
\begin{enumerate}[label=(\alph*),ref=(\alph*)]
    \item \label{cond:unconstrained:2} $|\x| \leq \delta$; and
    \item \label{cond:unconstrained:3} $\max B_p(\x) - \min B_p(\y) \leq \epsilon$ for all $\y \in \LL$. 
\end{enumerate}

Let $x^* \in \uu$ be a minimizer of \eqref{eq:POP_unconstrained}.
According to Corollary \ref{cor:cut-off}, there exists a subbox $\x \in \LL$ such that $x^* \in \x$.
To prove such $\x$ satisfies Condition \ref{cond:unconstrained:2}, notice from Remark \ref{rem:subdiv} that
\begin{equation}
    |\x| \leq 2^{-n} \leq 2^{-N} \leq \delta.
\end{equation}

To prove Condition \ref{cond:unconstrained:3}, first notice for any $\y \in \LL$, 
\begin{align}
    \min B_p(\y) & \geq \min_{y \in \y} p(y) - \zeta_p \cdot 2^{-2n} \label{eq:uncon:arg1} \\
    & \geq p(x^*) - \zeta_p \cdot 2^{-2n} \label{eq:uncon:arg2},
\end{align}
where
\eqref{eq:uncon:arg1} follows from Theorem \ref{thm:err_bound};
and \eqref{eq:uncon:arg2} follows from the definition of $x^*$.
Therefore
\begin{align}
    \max B_p(\x) - \min B_p(\y) & \leq \max_{x \in \x} p(x) - p(x^*) + 2 \zeta_p \cdot 2^{-2n} \label{eq:uncon:arg3} \\
    & \leq L_p \cdot \left( \sqrt{l} \cdot |\x| \right) + 2 \zeta_p \cdot 2^{-2n} \label{eq:uncon:arg4} \\
    & \leq \frac{\epsilon}{2} + \frac{\epsilon}{2} = \epsilon \label{eq:uncon:arg5}
\end{align}
where
\eqref{eq:uncon:arg3} follows from Corollary \ref{cor:err_bound} and \eqref{eq:uncon:arg2};
\eqref{eq:uncon:arg4} is true because $\|x^* - x \| \leq \sqrt{l} \cdot |\x|$ for all $x \in \x$;
and \eqref{eq:uncon:arg5} follows from \eqref{eq:uncon:K} and the assumption that $n > N$.
\end{proof}

According to \eqref{eq:uncon:K}, the rate of convergence with respect to (WRT) the decision variables is quadratic (1\ts{st} term); the rate of convergence WRT the objective function is either quadratic (2\ts{nd} term) or linear (3\ts{rd} term), depending on which term dominates.
However, a tighter bound exists if one of the global minimizers satisfies the second-order sufficient condition for optimality \cite[Theorem~2.4]{nocedal2006numerical}, which is shown in the following theorem:

\begin{thm}
\label{thm:uncon:roc2}
Let $p$ in \eqref{eq:POP_unconstrained} be a multivariate polynomial of dimension $l$.
Let the Hessian $\nabla^2 p$ be positive definite at some global minimizer $x^*$ of \eqref{eq:POP_unconstrained}, where $x^*$ is not on the boundary of $\uu$.
Then the maximum number of iterations needed to solve \eqref{eq:POP_unconstrained} up to accuracy $\epsilon$ and $\delta$ is
\begin{equation}
    N = \left\lceil \max\left\{
    -\log_2 \delta
    , \, C_1, \, -\frac{1}{2} \log_2 \epsilon + C_2 \right\} \right\rceil,\label{eq:max_iters_glob_min_not_on_bdry}
\end{equation}
where $C_1$ and $C_2$ are constants that only depend on $p$.
\end{thm}

\noindent Theorem \ref{thm:uncon:roc2} proves a quadratic rate of convergence WRT the objective function once the number of iterations exceeds a threshold $C_1$, given that the second-order optimality condition is satisfied at some global minimizer.
We now discuss the number of patches remaining after sufficiently many iterations, which gives an estimate of memory usage when Algorithm \ref{alg:pcba} is applied to solve \eqref{eq:POP_unconstrained}.

\begin{thm}
\label{thm:uncon:memory}
Suppose there are $m < \infty$ global minimizers $x^*_1, \cdots, x^*_m$ of \eqref{eq:POP_unconstrained},
and none of them are on the boundary of the unit box $\uu$.
Let the Hessian $\nabla^2 p$ be positive definite at these minimizers.
Then after sufficiently many iterations of Algorithm \ref{alg:pcba}, the number of Bernstein patches remaining (i.e., length of the list $\LL$ in Algorithm \ref{alg:pcba}) is bounded by a constant.
\end{thm}

\noindent The constant bound in Theorem \ref{thm:uncon:memory} scales exponentially with the problem dimension, and is a function of the condition number of the cost function's Hessian at the minimizers.

\subsection{Constrained Case}\label{sec:complexity_analysis:constrained}

\begin{thm}
\label{thm:roc:constrained}
Suppose that the linear independence constraint qualification (LICQ) \cite[Definition~12.4]{nocedal2006numerical} is satisfied at all global minimizers $x_1^*, \cdots x_m^*$ of the constrained POP \eqref{eq:POP}, and at least one constraint is active (i.e., the active set $\A(x^*)$ \cite[Definition~12.1]{nocedal2006numerical} is nonempty) at some minimizer $x^* \in \{x_1^*, \cdots, x_m^*\}$.
Then the maximum number of iterations needed to solve \eqref{eq:POP} up to accuracy $\epsilon$, $\delta$, and equality constraint tolerance $\eeq$ is 
\begin{equation}
\label{eq:roc:constrained_constants}
\begin{split}
    N := \Bigg\lceil \max \Bigg\{ & 
    C_7, \,
    -\log_2 \delta
    , \,
    - \log_2 \eeq + C_8, \,
    - \log_2 \epsilon + C_9
    \Bigg\} \Bigg\rceil,
\end{split}
\end{equation}
where $C_7$, $C_8$, $C_9$ are constants.
\end{thm}

\noindent Theorem \ref{thm:roc:constrained} gives a bound on the number of PCBA iterations needed to solve a POP up to specified tolerances.
In particular, \eqref{eq:roc:constrained_constants} shows the rate of convergence is
\emph{linear} in step tolerance (2\ts{nd} term), equality constraint tolerance (3\ts{rd} term), and objective function (4\ts{th} term),
once the number of iterations is larger than a constant (1\ts{st} term).
We next prove a bound on the number of items in the list $\LL$ after sufficiently many iterations.

\begin{thm}
\label{thm:memory_usage:constrained}
Suppose there are $m$ ($m < \infty$) global minimizers $x^*_1, \cdots, x^*_m$ of the constrained problem \eqref{eq:POP},
and none of them are on the boundary of the unit box $\uu$.
Let the critical cone (see \cite[(12.53)]{nocedal2006numerical}) be nonempty for \eqref{prog:relaxed_POP} as in the proof of Theorem \ref{thm:roc:constrained}.
Then after sufficiently many iterations of Algorithm \ref{alg:pcba}, the number of Bernstein patches remaining (i.e., length of the list $\LL$) is bounded by a constant.
\end{thm}

\noindent The constant proved in Theorem \ref{thm:memory_usage:constrained} scales exponentially with respect to the dimension of the problem.

\subsection{Memory Usage Implementation}\label{subsec:memory_usage}

We now state the amount of GPU memory required to store a single item $\itmx \in \LL$, given the degree and dimension of the cost and constraint polynomials.
Note that, for our implementation, all numbers in an item are represented using 4B of space, as either floats or unsigned integers.

For a multi-index $J = (j_1,\cdots,j_l) \in \N^l$, let $\Pi J = j_1\times\cdots\times j_l$, and let $J + n = (j_1 + n,\cdots,j_l + n)$ for $n \in \N$.
Let $P$ be the multi-degree of the cost $p$.
Let $G$ be a multi-degree large enough for all inequality constraints $g_i$, and $H$ a multi-degree large enough for all equality constraints $h_j$.
By ``large enough'' we mean that, if $G_i$ is the multi-degree of any $g_i$ then $G_i \leq G$ (and similarly for $H$).
Then, as per \cite[\S4.1]{titi2017fast}, an item can be stored in memory as an array with the following number of entries:
\begin{align}\label{eq:memory_per_item}
    2l + (\Pi(P+1)) + (\Nineq\cdot\Pi(G+1)) + (\Neq\cdot\Pi(H+1)),
\end{align}
where the first $2l$ entries store the upper and lower bounds (in each dimension) of the subbox $\x$.

\subsection{Summary}
We have shown that PCBA will find a solution to \eqref{eq:POP}, if one exists, in bounded time.
We have also shown that the memory usage of PCBA is bounded after a finite number of iterations, which implies that the memory usage is bounded; and we have provided a way to compute how much memory is required to store the list $\LL$.
Next, we benchmark PCBA on a variety of problems and compare it to two other solvers.
\section{PCBA Evaluation}\label{sec:benchmarking_PCBA}

In this section, we compare PCBA against a derivative-based solver (MATLAB's \fmincon\ \cite{MatlabOTB}) and a convex relaxation method (Lasserre's BSOS \cite{lasserre2017bounded}).
First, we test all three solvers on eight \textbf{Benchmark Evaluation} problems with dimension less than or equal to 4.
Second, we compare all three solvers on several \textbf{Increasing Number of Constraints} problems, to assess how each solver scales on a variety of difficult objective functions \cite{gavana2019testsuite}.

All of the solvers/problems in this section are run on a computer with a 3.7GHz processor, 16 GB of RAM, and an Nvidia GTX 1080 Ti GPU.
PCBA is implemented with MATLAB R2017b executables and CUDA 10.2.
Our code is available on GitHub: \texttt{\url{https://github.com/ramvasudevan/GlobOptBernstein}}.

\subsection{Parameter Selection}
To set up a fair comparison, we scale each problem to the $\uu = [0,1]^l$ box, where $l$ is the problem dimension.
For PCBA, we use the stopping criteria in \S\ref{sec:constrained_BA}.
To choose $\epsilon$, we first compute the patch $B(\uu)$, then set
\begin{align}\label{eq:benchmarking_epsilon}
    \epsilon = (10^{-7})\cdot\left(\max B(\uu) - \min B(\uu)\right).
\end{align}
We set the maximum number of iterations to $N = 28$.
We do not set $\delta$, which determines the minimum number of iterations; $\delta$ is only needed to prove complexity bounds in \S\ref{sec:complexity_analysis}.

BSOS \cite[\S 4]{lasserre2017bounded} requires the user to specify the size of the semidefinite matrix associated with the convex relaxation of the POP.
This is done by by selecting a pair of parameters, $d$ and $k$ (note these are different from our use of $d$ and $k$). 
Though one has to increase $d$ and $k$ gradually to ensure convergence, larger values of $d$ and $k$ correspond to larger semidefinite programs, which can be difficult to solve.
We chose $d$ and $k$ separately for the Benchmark Evaluation.
We used $d = k = 2$ for the Increasing Number of Constraints.

For \fmincon~\cite{MatlabOTB}, we set the \texttt{OptimalityTolerance} option to $\epsilon$ in \eqref{eq:benchmarking_epsilon}.
We set $\texttt{MaxFunctionEvaluations} = 10^5$ and $\texttt{MaxIterations} = 10^4$.
We also provide \fmincon\ with the analytic gradients of the cost and constraints.

\subsection{Benchmark Evaluation}

\subsubsection{Setup} We tested PCBA, BSOS, and \fmincon\ on eight benchmark POPs \cite{nataraj2011constrained}, listed as P1 through P8; the problems are reported in the Appendix.
We ran each solver 50 times on each problem, and report the median solution error and time required to find a solution.
Since {\fmincon} may or may not converge to the global optimum depending on its initial guess, we used random initial guesses for each of the 50 attempts.

\subsubsection{Results}
The results are summarized in Table \ref{tab:benchmark_results}.
For additional results (e.g., to plot the results of any of the problems), see the \href{https://github.com/ramvasudevan/GlobOptBernstein}{\underline{\blue{GitHub repository}}}.

In terms of solution quality, PCBA always found the solution to within the desired optimality tolerance $\epsilon$, except for on P1, where PCBA stopped at the maximum allowed number of iterations (28); PCBA always used between 22 and 28 iterations.
BSOS always found a lower bound to the solution, as expected.
While {\fmincon} converged to the global optimum at least once on every problem, it often converged to local minima, hence the large error values on some problems.
In terms of solve time, \fmincon\ solves the fastest (in 10--20 ms), PCBA is about twice as slow as \fmincon, and BSOS is one to two orders of magnitude slower than PCBA.

For PCBA, the memory usage (computed with \eqref{eq:memory_per_item}) increases roughly by one order of magnitude for each additional dimension of the decision variable increases (Table \ref{tab:benchmark_results} reports the peak GPU memory used by PCBA on each benchmark problem).
Notice that PCBA never uses more than several MB of GPU memory, which is much less than the 11 GB available on the Nvidia GTX 1080 Ti GPU.
Figure \ref{fig:num_patches_vs_num_iterations} shows the number of patches and the amount of GPU memory used on P4.
We see that the memory usage peaks, then stays below a constant, as predicted by Theorem \ref{thm:memory_usage:constrained}.

\begin{figure}[t]
    \centering
    \includegraphics[width=\columnwidth]{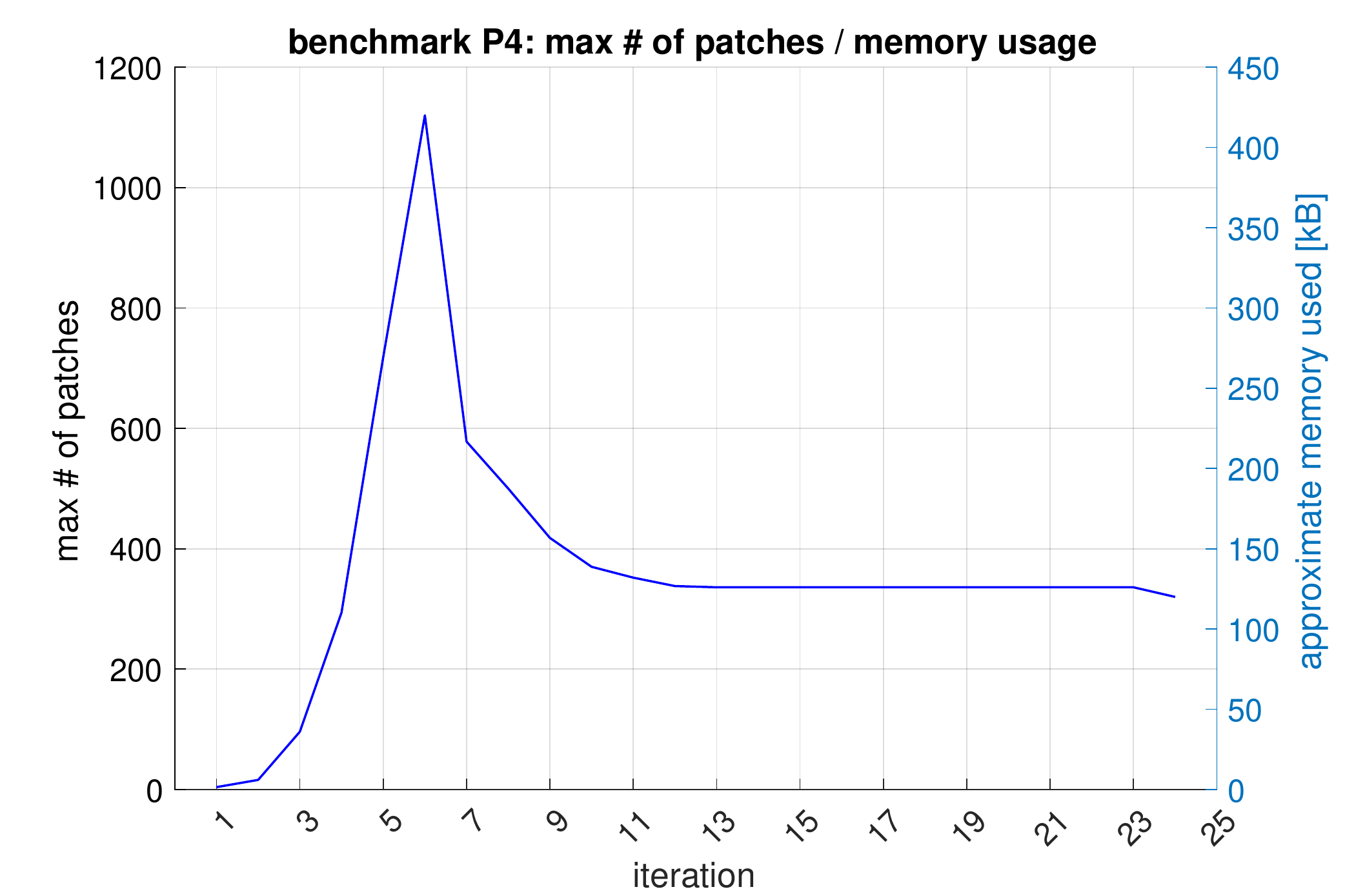}
    \caption{The maximum number of patches (left axis) and corresponding GPU memory used (right axis) at each iteration of PCBA, for P4 of the benchmark problems (see \S\ref{sec:benchmarking_PCBA}).
    This problem took 24 iterations to solve.
    Notice that the number of patches peaks in iteration 5, then stays under 400 patches at every iteration from iteration 9 onwards; this visualizes Theorem \ref{thm:memory_usage:constrained}.}
    \label{fig:num_patches_vs_num_iterations}
\end{figure}

\begin{table*}[t]
\centering
\begin{tabular}{cr|r|r|r|c|r|r|r|r|r|r}
\multirow{2}{*}{} & & \multicolumn{5}{c|}{PCBA} & \multicolumn{3}{c|}{BSOS \cite{lasserre2017bounded}} & \multicolumn{2}{c}{{\fmincon} \cite{MatlabOTB}} \\ \cline{3-12} 
 \multicolumn{1}{c|}{} & \multicolumn{1}{c|}{$l$} &\multicolumn{1}{c|}{$\epsilon$}& \multicolumn{1}{c|}{error} & \multicolumn{1}{c|}{time {[}s{]}} & iterations & memory & $(d,k)$ & \multicolumn{1}{c|}{error} & \multicolumn{1}{c|}{time {[}s{]}} & \multicolumn{1}{c|}{error} & \multicolumn{1}{c}{time {[}s{]}} \\ \hline
\multicolumn{1}{c|}{P1} & \multicolumn{1}{c|}{2} & 7.0000e-07&7.9002e-07 & 0.0141& 28 & 23 kB &(3,3) & -0.0068 & 1.2264 & \multicolumn{1}{r|}{1.0880} & 0.0083 \\
\multicolumn{1}{c|}{P2}& \multicolumn{1}{c|}{2} &0.1369 &0.0731& 0.0131 & 26 & 60 kB &(2,2) & -3.4994e-04 & 0.7671 & \multicolumn{1}{r|}{ -0.4447} & 0.0180 \\
\multicolumn{1}{c|}{P3}& \multicolumn{1}{c|}{2} &2.0000e-06 &1.9879e-06 & 0.0128 & 26 & 57 kB & (2,2) & -3.2747 &0.5065 & \multicolumn{1}{r|}{-3.0000} & 0.0220 \\
\multicolumn{1}{c|}{P4}& \multicolumn{1}{c|}{3} &1.7000e-06 &6.5565e-07 & 0.0204 & 24 & 416 kB &(4,4) & -0.9455 & 6.6546 & \multicolumn{1}{r|}{3.2000e-05} & 0.0130 \\
\multicolumn{1}{c|}{P5}& \multicolumn{1}{c|}{3} &1.0000e-06 &0 & 0.0312 & 28 & 3 MB &(2,2) & -4.4858e-07 & 0.8462 & \multicolumn{1}{r|}{7.9985e-06} &0.0062\\
\multicolumn{1}{c|}{P6}& \multicolumn{1}{c|}{4} &4.2677e-04 &1.8685e-04 & 0.0315 & 23 & 2 MB & (3,3) & -36.6179 & 6.2434 & \multicolumn{1}{r|}{ 0.0040} & 0.0065\\
\multicolumn{1}{c|}{P7} & \multicolumn{1}{c|}{4} &5.0000e-07 &2.5280e-07 &0.0374& 26 & 282 kB &(1,1) & -1.0899 & 0.1839 & \multicolumn{1}{r|}{ 2.4002e-07} &0.0139\\
\multicolumn{1}{c|}{P8} & \multicolumn{1}{c|}{4} &1.3445e-04 &6.7803e-05 & 0.0440 & 22 & 6 MB & (2,2) & -3.2521 & 1.8295 & \multicolumn{1}{r|}{9.9989e-04} &  0.0169
\end{tabular}
\caption{
Results for PCBA, BSOS, and {\fmincon} on eight benchmark problems with 2, 3, and 4 dimensional (column $l$) decision variables (see the Appendix for more details).
The error columns report each solver's result minus the true global minimum.
For all three solvers, the reported error and time to find a solution are the median over 50 trials (with random initial guesses for \fmincon).
For PCBA, we also report the optimality tolerance $\epsilon$ (as in \eqref{eq:benchmarking_epsilon}), number of iterations to convergence, and peak GPU memory used.
Note that, on P1 and P5, PCBA stopped at the maximum number of iterations (28).}
\label{tab:benchmark_results}
\vspace*{-0.5cm}
\end{table*}

\subsection{Increasing Constraint Problems}
\label{subsec:example:constraints}

Next, we tested each solver on problems with an increasing number of constraints, to each solver for use with RTD; we find in practice that a robot running RTD must handle between 30 and 300 constraints at each planning iteration.

\subsubsection{Setup}
We first choose an objective function with either many local minima or a nearly flat gradient near the global optimum (the global optimizer is known for each function).
In particular, we tested on the ElAttar-Vidyasagar-Dutta, Powell, Wood, Dixon-Price (with $l = 2, 3, 4$), Beale, Bukin02, and Deckkers-Aarts problems (see the Appendix and \cite{gavana2019testsuite}).

For each objective function, we generate 200 random constraints in total, while ensuring at least one global optimizer stays feasible (if there are multiple global optimizers, we choose one at random that will be feasible for all constraints).
To generate a single constraint $g: \uu \to \R$, we first create a polynomial $g_{\regtext{temp}}$ as a sum of the monomials of the decision variable with maximum degree 2, with random coefficients in the range $[-5,5]$.
To ensure $x^*$ is feasible, we evaluate $g_{\regtext{temp}}$ on $x^*$, then subtract the resulting value from $g_{\regtext{temp}}$ to produce $g$ (i.e., $g \leftarrow g_{\regtext{temp}} - g_{\regtext{temp}}(x^*)$).

We run PCBA, BSOS, and \fmincon\ on each objective function for 20 trials, with 10 random constraints in the first trial, and adding 10 constraints in each trial.
As before, we run \fmincon\ 50 times for each trial with random initial guesses, since its performance is dependent upon the initial guess.

\subsubsection{Results}
To illustrate the results, data for the Powell objective function are shown in Figure \ref{fig:increasing_constraint_results}.
The data (and plots) for the other objective functions are available in the \href{https://github.com/ramvasudevan/GlobOptBernstein}{\underline{\blue{GitHub repository}}}

In terms of solution quality, all three algorithms converge to the global optimum often when the number of constraints is low, but {\fmincon} converges to suboptimal solutions more frequently as the number of constraints increases.
PCBA and BSOS are always able to find the optimal solution.
PCBA is always able to find the global optimum regardless of the number of constraints, unlike BSOS (which runs out of memory) or {\fmincon} (which converges to local minima).

All three solvers require an increasing amount of solve time as the number of constraints increases.
PCBA is comparable in speed to {\fmincon} on 2-D problems, but is typically slower on higher-dimensional problems.
Regardless of the number of constraints, BSOS takes three to four orders of magnitude more time to solve than PCBA or {\fmincon}.

More details on PCBA are presented in Table \ref{tab:increasing_constraints}.
PCBA's time to find a solution increases roughly by an order of magnitude when the decision variable dimension increases by 1; however, PCBA solves all of the increasing constraint POPs within 0.5 s.
The memory usage increases by 1--3 orders of magnitude with each additional dimension; however, PCBA never uses more than 650 MB of GPU memory, well below the 11 GB available.
Figure \ref{fig:memory_vs_num_constraints} shows PCBA's GPU memory usage versus the number of constraints for the Powell objective function.

\subsection{Summary}
As expected from the complexity bounds in \S\ref{sec:complexity_analysis:constrained}, the results in this section indicate that PCBA is practical for quickly solving 2-D POPs with hundreds of constraints.
We leverage this next by applying PCBA to solve RTD's POP \eqref{prog:online_trajopt} for real-time receding-horizon trajectory optimization.

\begin{figure}[t]
    \centering
    \begin{subfigure}[t]{\columnwidth}
        \centering
        \includegraphics[width=\columnwidth]{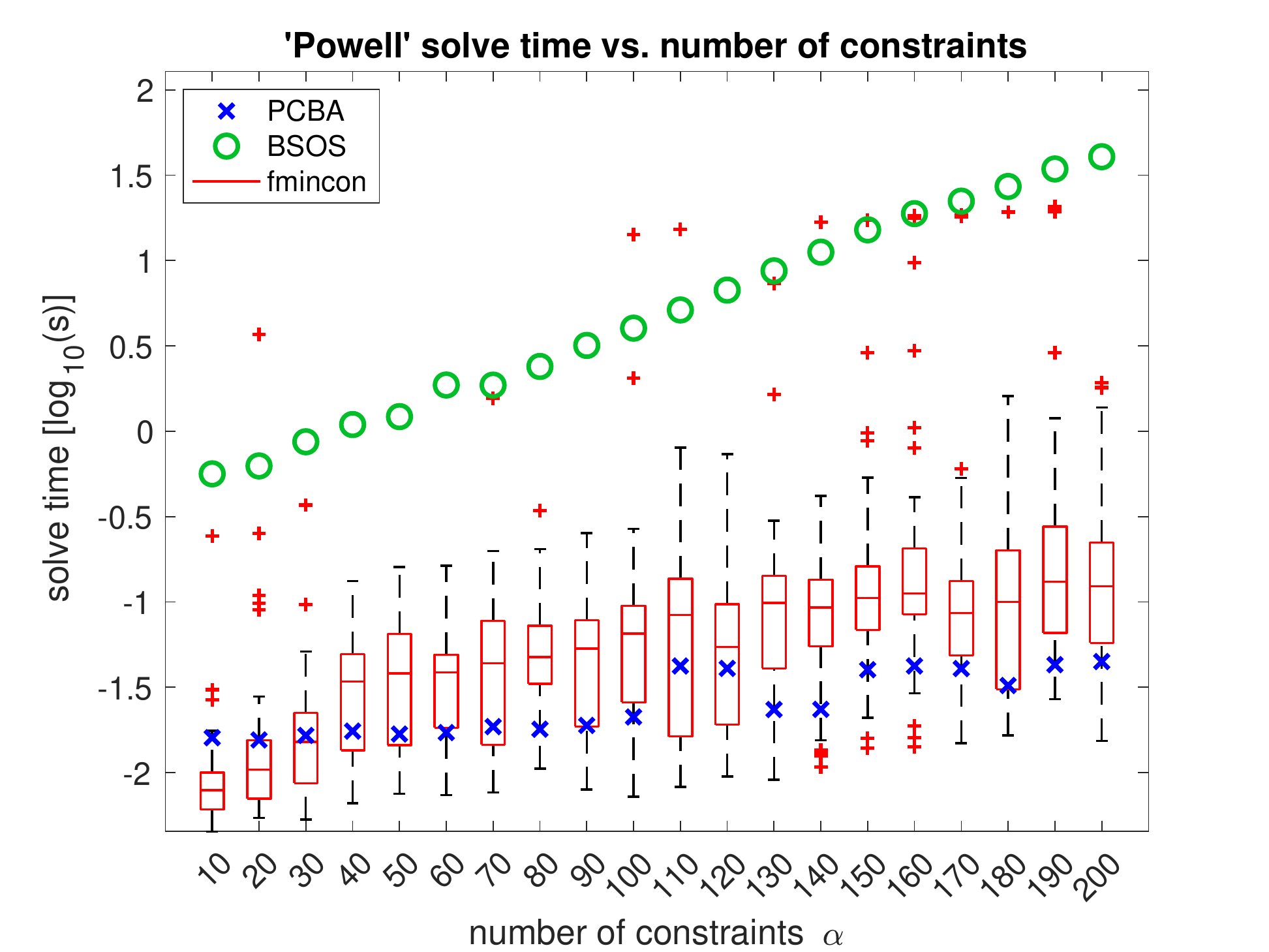}
    \end{subfigure}
    
    \begin{subfigure}[t]{\columnwidth}
        \centering
        \includegraphics[width=\columnwidth]{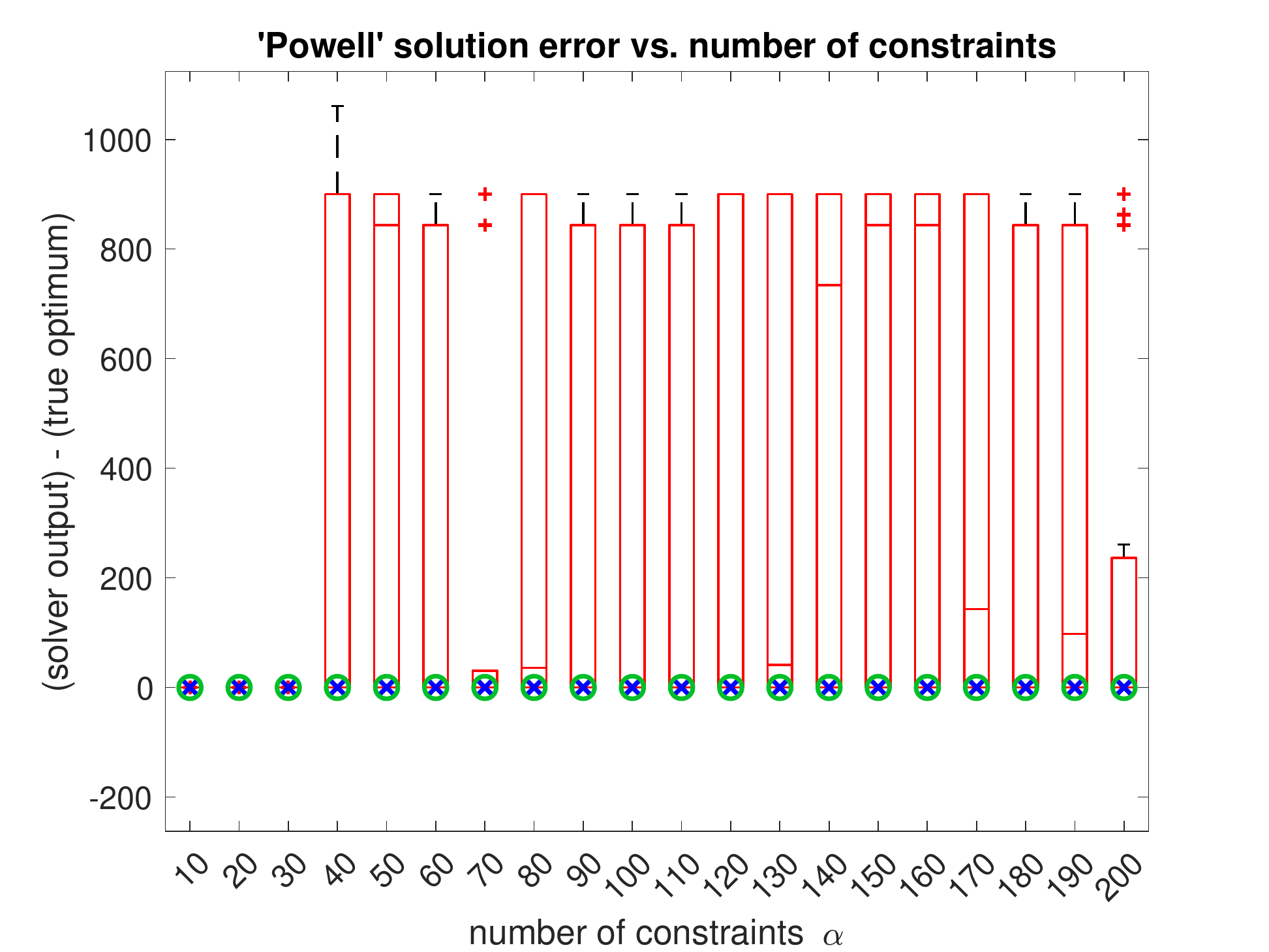}
    \end{subfigure}
    \caption{
    Results for an increasing number of constraints on the Powell objective function (see the Appendix) for the PCBA, BSOS, and \fmincon\.
    The top plot shows the time required to solve the problem as the number of constraints increases.
    The bottom plot shows the error between each solver's solution and the true global optimum.
    For both time and error, {\fmincon} is shown as a box plot over 50 trials with random initial guesses; the central red line indicates the median, the top and bottom of the red box indicate the 25\ts{th} and 75\ts{th} percentiles, the black whiskers are the most extreme values not considered outliers, and the outliers are red plus signs.
    PCBA solves the fastest in general; \fmincon\ typically solves slightly slower than PCBA for more than 40 constraints; and BSOS is the slowest solver.
    PCBA and BSOS always find the global optimum, as does \fmincon\ when there are not many constraints, because the Powell objective function is convex.
    Above 30 constraints, \fmincon\ frequently has large error due to convergence to local minima.
    }
    \label{fig:increasing_constraint_results}
\end{figure}

\begin{figure}[ht]
    \centering
    \includegraphics[width=\columnwidth]{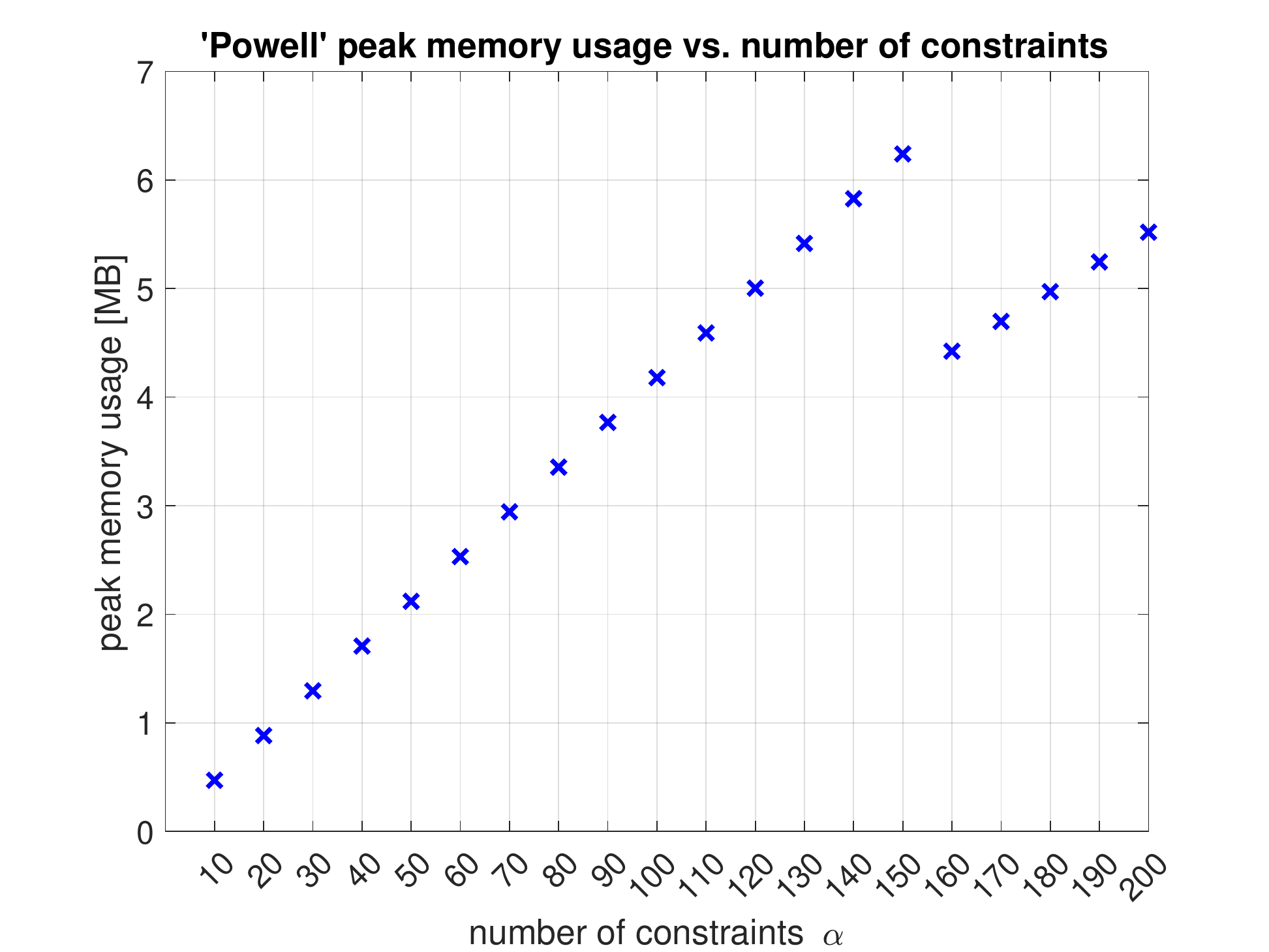}
    \caption{The approximate peak GPU memory used by PCBA for the Powell problem, as a function of the number of constraints.
    Since the amount of memory required per item in the list $\LL$ grows linearly with the number of constraints, the overall memory usage also grows linearly.
    However, at 160 constraints, we see a drop in the memory usage; this is because the additional constraints render more parts of the problem domain infeasible, resulting in more items being eliminated per PCBA iteration.
    Note that the maximum memory usage is well under the several GB of memory available on a typical GPU.}
    \label{fig:memory_vs_num_constraints}
\end{figure}

\begin{table}[ht]
\centering
\begin{tabular}{c|c|r|r|r}
\multicolumn{1}{l|}{} & $l$ & \multicolumn{1}{c|}{max time {[}s{]}} & \multicolumn{1}{c|}{max items} & \multicolumn{1}{c}{max memory} \\ \hline
E-V-D & 2 & 0.0360 & 66 & 171 kB \\
Powell & 2 & 0.0447 & 1200 & 6.24 MB \\
Wood & 2 & 0.0532 & 54 & 220 kB \\
D-P 2-D & 2 & 0.0259 & 90 & 433 kB \\
D-P 3-D & 3 & 0.0675 & 356 & 3.47 MB \\
D-P 4-D & 4 & 0.402 & 4994 & 193 MB \\
Beale & 2 & 0.0302 & 106 & 259 kB \\
Bukin02 & 4 & 0.393 & 10550 & 110 MB \\
D-A & 4 & 0.389 & 51886 & 647 MB
\end{tabular}
\caption{
Results for the increasing constraints PCBA evaluation.
Abbreviated problem names (as in the Appendix) are on the left, along with each problem's decision variable dimension $l$.
Over all 20 trials (with between 10 and 200 constraints), we report the maximum time spent find a solution, the maximum number of items in the list $\LL$, and the maximum amount of GPU memory used.
Note that the problems all solved under 0.5 s regardless of the number of constraints, and no problem requested more than 650 MB of memory.}
\label{tab:increasing_constraints}
\end{table}
\section{Hardware Demonstrations}\label{sec:hardware_demo}

Recall from \S\ref{sec:preliminaries:RTD} that RTD enables real-time, provably collision-free trajectory planning via solving a POP every planning iteration.
RTD is provably collision-free regardless of the POP solver used, meaning that we are able to test PCBA safely on hardware; when PCBA is applied to RTD, we call the resulting trajectory planning algorithm \textbf{RTD*}.
See Figure \ref{fig:time_lapse} and the \href{https://youtu.be/YcH4WAzqPFY}{\blue{\underline{video}} (click for link)}.

In this section, we apply RTD* to a Segway robot navigating a hallway with static obstacles.
Recall that RTD always produces \emph{dynamically feasible} trajectory plans \cite{kousik2018bridging}.
As proven in Corollary \ref{cor:cut-off} and demonstrated in \S\ref{sec:benchmarking_PCBA}, PCBA always finds \emph{optimal} solutions.
This section shows that PCBA/RTD* improves the \emph{liveness} of a robot by successfully navigating a variety of scenarios, and outperforming \fmincon/RTD.

\subsection{Overview}
\subsubsection{Demonstrations}
We ran two demonstrations in a $20 \times 3$ m$^2$ hallway.
In the first demonstration, we filled the hallway with random static obstacles and ran RTD*.
In the second demonstration, we constructed two difficult scenarios and ran bot PCBA/RTD* and \fmincon/RTD on each.

\subsubsection{Hardware}
We use a Segway differential-drive robot with a planar Hokuyo UTM-30LX LIDAR for mapping and obstacle detection (see Figure \ref{fig:time_lapse}).
Mapping, localization, and trajectory optimization run onboard, on a 4.0 GHz laptop with an Nvidia GeForce GTX 1080 GPU.

\subsubsection{POPs}
As in \S\ref{sec:preliminaries:RTD} (also see Figure \ref{fig:time_lapse}), the robot plans by optimizing over a set $Q \subset \R^2$ of parameterized trajectories as in \cite[(16)]{kousik2018bridging}.
The parameters are speed $q_1 \in [0, 1]$ m/s and yaw rate $q_2 \in [-1, 1]$ rad/s, so $Q = [0,1]\times[-1,1]$.
The trajectories are between 1 and 2 s long, depending on the robot's initial speed (since each trajectory includes a braking maneuver).
The robot must find a new plan (i.e., PCBA must solve a new POP) every 0.5 s, or else it begins executing the braking maneuver associated with its previously-computed plan \cite[Remark 70]{kousik2018bridging}.
In other words, we require PCBA to return a feasible solution or that the problem is infeasible.
PCBA is given a time limit of 0.4 s to find a solution, because the robot requires 0.1 s for other onboard processes.

At each planning iteration, obstacles are converted into constraints as in \cite[Sections 6 and 7]{kousik2018bridging}; this produces a list of $\Nineq \in \N$ inequality constraints, which we denote $w(x_1,\cdot), \cdots, w(x_\Nineq,\cdot): Q \to \R$.
In practice, $30 \leq \Nineq \leq 300$ (see Figure \ref{fig:number_of_constraints_per_trial}).

The objective function is constructed at each planning iteration as follows.
Offline, we represent the endpoint of any parameterized trajectory as a degree 10 polynomial $x: Q \to X$.
At runtime, we generate $N_{\regtext{wp}} \in \N$ waypoints (i.e., desired locations for the robot to reach), denoted $\{x_n\}_{n = 1}^{N_{\regtext{wp}}}$.
We then create the following POP:
\begin{align}\label{prog:online_trajopt_POP}
    \begin{array}{cl}
         \underset{q\,\in\,Q}{\regtext{argmin}} & \prod_{n=1}^{N_{\regtext{wp}}}\left(\norm{x(q) - x_n}_2^2\right) \\
         \regtext{s.t} & w(x_i,q) \leq 0~\forall~i = 1,\cdots,\Nineq.
    \end{array}
\end{align}
The objective function is degree $10\cdot N_{\regtext{wp}}$, and the constraints are each degree 12.
Furthermore, ignoring the constraints, the objective function requires the POP solver to choose between as many global minima as there are waypoints.

\subsection{Demo 1}

The first demo shows the ability of the RTD* to plan safe trajectories in randomly-generated scenarios in real time, demonstrating \emph{dynamic feasibility}, \emph{optimality}, and \emph{liveness}.

\subsubsection{Setup}
The robot was required to move autonomously back and forth ten times between two global goal locations spaced $12$ m apart, while ($30$ cm)\ts{3} box-shaped obstacles were randomly placed in its path.
At each planning iteration, we generated $N_{\regtext{wp}} = 2$ waypoints, $x_{\regtext{L}}$ and $x_{\regtext{R}}$, both 1.5 m ahead of the robot in the direction of the global goal; $x_{\regtext{L}}$ is on the left side of the hallway relative to the robot, and $x_{\regtext{R}}$ is on the right.

After running the robot with RTD*, we ran {\fmincon} on the 528 saved POPs generated during these 10 trials (we do not run BSOS due to its slow solve time).
Each POP has between 49 and 245 constraints (see Figure \ref{fig:number_of_constraints_per_trial}).
For each POP, we initialized {\fmincon} with 25 random initial guesses, and did not require {\fmincon} to solve within 0.4 s (i.e., we did not enforce the real time planning constraint).
To understand the timing of PCBA and for fair comparison with \fmincon, we re-ran PCBA 25 times on each trial and did not require it to solve in real time.

\subsubsection{Results}
The robot running RTD* successfully completed every trial (meaning that it reached the global desired goal location without collisions, and without human assistance).

When re-running on the saved POPs, \fmincon\ performs nearly as well as PCBA in terms of finding solutions.
Out of all $25\times528$ attempts in which \fmincon\ converged to a feasible solution, \fmincon\ converged to a greater cost than PCBA 93.9\% of the time (recall that PCBA provably upper-bounds the optimal solution); however, the \fmincon\ solution was only 0.77\% greater in cost than the PCBA solution on average, indicating that \fmincon\ was often able to find a global optimum when given enough attempts.
In terms of feasibility, PCBA and \fmincon\ also show similar results.
PCBA reports that 7.01\% of the POPs are infeasible, whereas \fmincon\ converged to an infeasible result on 8.08\% of the $25\times 528$ total attempts.
Note that, on 14.2\% of the 528 POPs, \fmincon\ converged to an infeasible result least once out of 25 attempts.

Where \fmincon\ suffers with respect to PCBA is in its consistency of finding an answer within the time limit (see Figure \ref{fig:pcba_vs_fmincon_solve_time}).
While \fmincon\ is often able to solve in $10^{-2}$ s (an order of magnitude faster than PCBA), it has a standard deviation of up to $10$ s.
On the other hand, PCBA always finds a solution or returns infeasible within 0.4 s, and has a standard deviation of 2.4 ms on average over all $25\times528$ POPs.
To summarize: as we expect from the theory in \S\ref{sec:complexity_analysis}, PCBA's solve time in practice appears constant \emph{on real trajectory optimization problems}.

\subsection{Demo 2}

The second demo shows that RTD* can navigate difficult scenarios because PCBA is able to rapidly solve POPs with hundreds of constraints.
Recall that, in any planning iteration, RTD and RTD* command the robot to begin braking if they cannot find a new trajectory plan (i.e., solve \eqref{eq:POP}).
By \emph{difficult} scenarios, we mean that the obstacles are arranged to cause the robot to have to brake often.
Therefore, by RTD*'s successful navigation of these scenarios, we demonstrate \emph{liveness}.

\subsubsection{Setup}
The robot was required to navigate two difficult scenarios autonomously.
In the first scenario, static obstacles were arranged to force the robot to turn frequently, and to decide to go left or right around each obstacle.
In the second scenario, the robot was required to navigate a tight obstacle blockade.
For each scenario, we ran PCBA/RTD* and \fmincon/RTD once each.
At each planning iteration, we generate $N_{\regtext{wp}} = 1$ waypoint positioned 1.5 m away from the robot along a straight line to the global goal; this produces a convex cost function for \eqref{prog:online_trajopt_POP}, but the constraints make the problem nonconvex.

\subsubsection{Results}
In the first scenario, both RTD* and RTD are able to successfully navigate the scenario.
Recall that the robot begins emergency braking when it does not find a feasible trajectory in a planning iteration.
RTD* brakes 6 times, whereas RTD brakes 13 times; furthermore, RTD* only takes 27 s to navigate to the goal, whereas RTD takes 43 s.
In other words, PCBA/RTD* is half as conservative as \fmincon/RTD, because it finds feasible solutions more frequently.

The results of the second scenario confirm that RTD* is less conservative than RTD.
In this scenario, RTD* is able to navigate the entire scenario autonomously without human assistance, whereas RTD causes the robot to become stuck (see Fig. \ref{fig:fmincon_stuck}), and a human operator must drive the robot for a short time to enable to it to continue moving autonomously.

\begin{figure}
    \centering
    \includegraphics[width=\columnwidth]{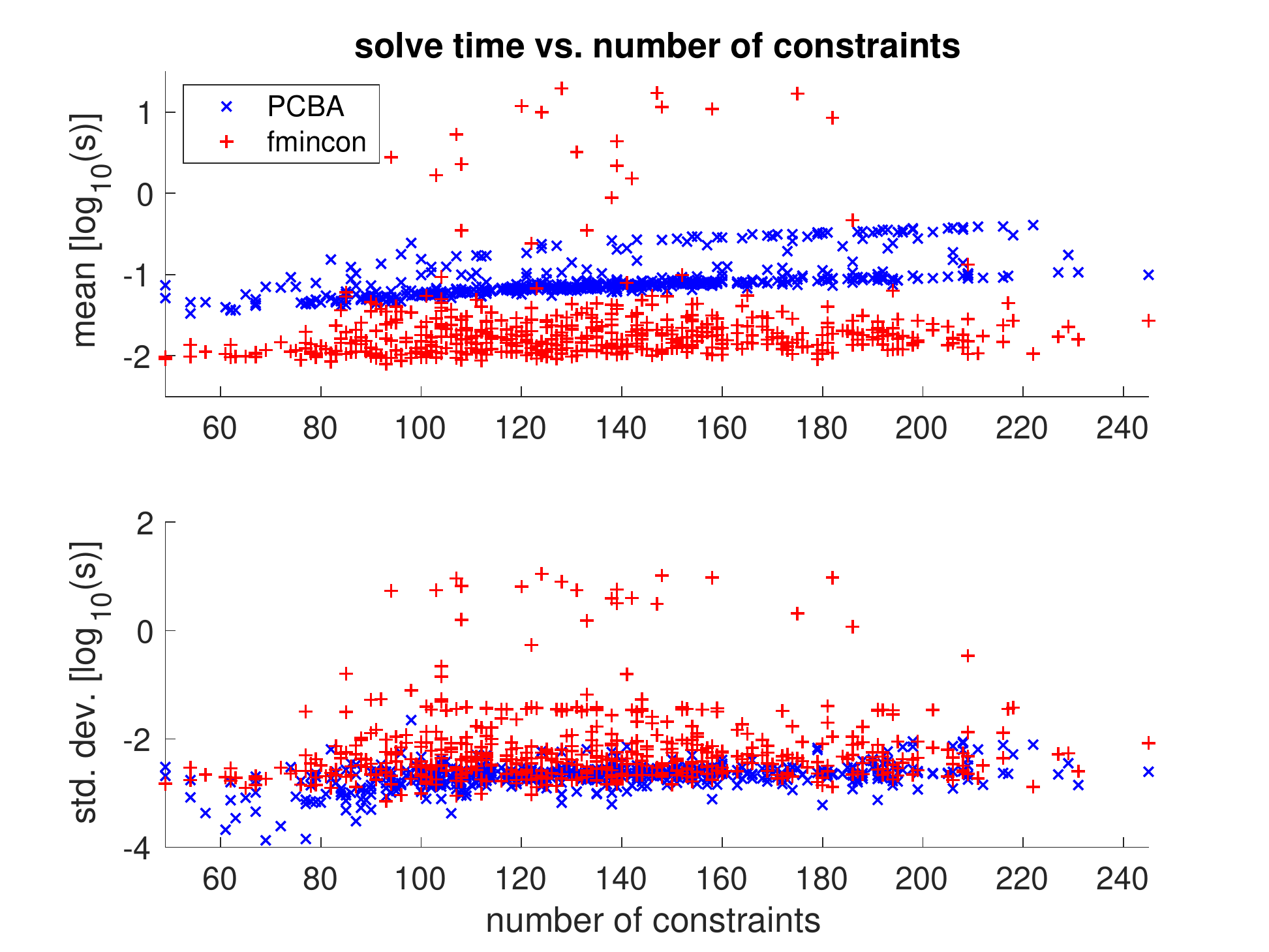}
    \caption{Solve times of PCBA and \fmincon\ on $528$ POPs generated by the Segway robot navigating random scenarios in Demo 1.
    Each POP was solved 25 times by each solver.
    While \fmincon\ can often find a solution an order of magnitude faster than PCBA, it also has a much higher standard deviation, meaning that it is less consistent at obeying the real-time limit required by mobile robot trajectory planning.}
    \label{fig:pcba_vs_fmincon_solve_time}
\end{figure}

\begin{figure}
    \centering
    \includegraphics[width=0.9\columnwidth]{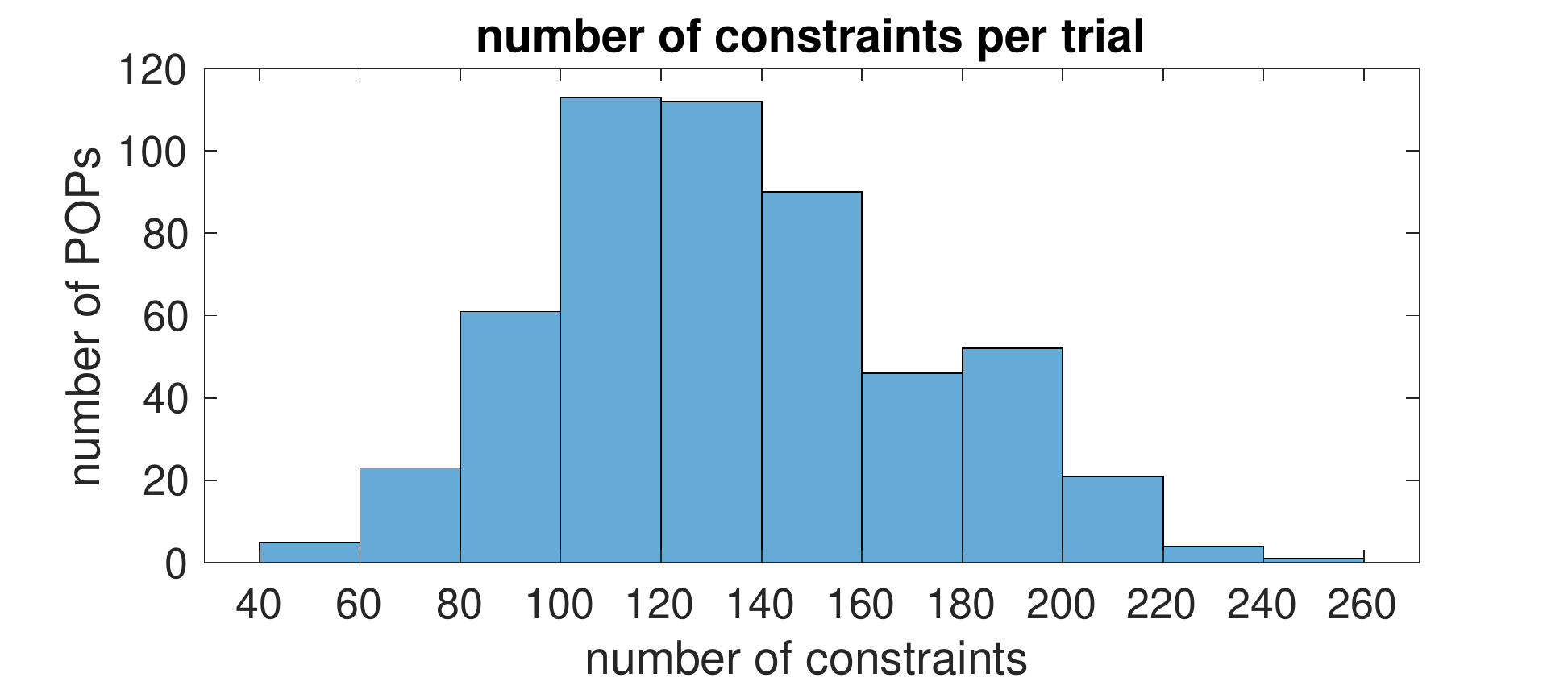}
    \caption{The number of POPs from Demo 1, out of $528$, that fall into the given bins of number of constraints; we see that most of the POPs had 100 -- 140 constraints.
    This number of constraints can makes it challenging to solve a POP while constrained by a real-time planning limit.}
    \label{fig:number_of_constraints_per_trial}
\end{figure}

\subsection{Discussion}
We have demonstrated that RTD* is capable of \emph{dynamic feasibility}, \emph{optimality}, and \emph{liveness} for online trajectory optimization on robot hardware\textbf{}.
RTD* is able to find an optimal solution, if it exists, at every receding-horizon planning iteration, leading to it consistently navigating random scenarios without collisions.
Furthermore, RTD* outperforms RTD at the same navigation and obstacle-avoidance tasks.
This performance increase is due to PCBA's ability to find solutions more quickly than \fmincon\ on problems with hundreds of constraints.
To the best of our knowledge, this is the first time any Bernstein algorithm has been demonstrated as practical for a real-time mobile robotics application.

\begin{figure}
    \centering
    \includegraphics[width=\columnwidth]{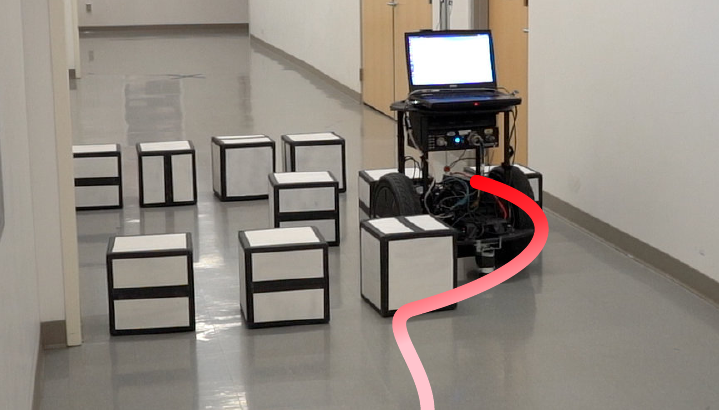}
    \caption{The robot becomes stuck when planning with RTD/\fmincon\ in the second scene of the second hardware demo, because \fmincon\ cannot find an optimal solution quickly enough given the high number of constraints produced by the surrounding obstacles.
    The robot requires human assistance to proceed, whereas it is able to navigate the entire scene autonomously when planning with RTD*/PCBA (Figure \ref{fig:time_lapse}).
    See the \href{https://youtu.be/YcH4WAzqPFY}{\blue{\underline{video}}}.}
    \label{fig:fmincon_stuck}
\end{figure}
\section{Conclusion}\label{sec:conclusion}

Mobile robots typically use receding-horizon planning to move through the world.
Plans should be dynamically feasible and optimal, but also need to be generated quickly, otherwise the robot may stop frequently and never complete a user-specified task.
We propose and implement a Parallel Constrained Bernstein Algorithm (PCBA) to rapidly generate provably dynamically-feasible and optimal plans.
PCBA takes advantage of the polynomial optimization problem (POP) structure used by the existing Reachability-based Trajectory Design (RTD) method; but, where RTD uses a gradient-based nonlinear solver that cannot guarantee optimality or real-time performance, we use PCBA.
The resulting algorithm, RTD*, is shown to outperform RTD on a variety of hardware demonstrations.
To the best of our knowledge, this is the first time a Bernstein Algorithm has been used for real-time mobile robotics.
Furthermore, PCBA outperforms the Bounded Sums-of-Squares (BSOS) and \fmincon\ solvers on a variety of benchmark POPs.
For future work, we plan to explore non-polynomial optimization problems and to improve the proposed time and space complexity bounds of PCBA; our goal is to apply RTD* and PCBA to more types of robots.

\renewcommand{\bibfont}{\normalfont\small}
{\renewcommand{\markboth}[2]{}
\printbibliography}

\begin{appendices}
\section{Proofs}\label{app:proofs}

\subsection{Theorem \ref{thm:uncon:roc2} (Unconstrained Rate of Convergence)}

\begin{proof}
The first term in \eqref{eq:max_iters_glob_min_not_on_bdry} follows from Theorem \ref{thm:roc:unconstrained} and guarantees $|\x| \leq \delta$.
To prove $\max B_p(\x) - \min B_p(\y) \leq \epsilon$ for all $\y \in \LL$ using the second and third terms, we need to define $C_1$ and $C_2$.
First, we use the fact that $x^*$ is not on the boundary of $\uu$ to find $R > 0$ such that $\x \subset \B_R(x^*)$ (the closed Euclidean-norm ball centered at $x^*$ with radius $R$).
Using the first-order necessary condition for optimality \cite[Theorem~2.2]{nocedal2006numerical} we know $\nabla p(x^*) = 0$.
Let $\sigma_{\max}$ be the maximum singular value of $\nabla^2 p(x^*)$.
Consider an arbitrary point $x \in \uu$ and let $t := x - x^*$.
Using Taylor's theorem, we have
\begin{equation}
\label{eq:uncon:taylor}
    p(x) = p(x^*) + \nabla p(x^*)^\top \, t + \frac{1}{2} t^\top \nabla^2 p(x^*) \, t + o(\|t\|^2)
\end{equation}
Since $\nabla p(x^*) = 0$, \eqref{eq:uncon:taylor} implies that $p(x) - p(x^*) \leq \frac{1}{2} \sigma_{\max} \|t\|^2 + o(\|t\|^2)$.
For sufficiently small $t$, the second order term dominates higher order terms.
That is, there exists $R > 0$ such that
\begin{equation}
\label{eq:uncon:approx}
    p(x) - p(x^*) \leq \sigma_{\max} \|t\|^2
\end{equation}
for all $x \in \B_R(x^*)$.

Now we use $R$ and $\sigma_{\max}$ to construct $C_1$ and $C_2$.
Let $n$ be the current iteration number such that
\begin{equation}
\label{eq:uncon:req}
    n \geq \left\lceil \max \left\{ -\log_2 \left( \frac{R}{\sqrt{l}} \right), -\frac{1}{2} \log_2 \left( \frac{\epsilon}{\sigma_{\max} l + 2\zeta_p} \right)  \right\} \right\rceil.
\end{equation}
Notice $\x \subset \B_R(x^*)$ because $\max_{x \in \x} \|x - x^*\| \leq |\x| = \sqrt{l} \cdot 2^{-n} \leq R$.
Therefore, \eqref{eq:uncon:approx} is satisfied for all $x \in \x$.
We also have
\begin{align}
    \max B(\x) - \min B(\y) &\leq \max_{x \in \x}p(x) - p(x^*) + 2\zeta_p \cdot 2^{-2n} \label{eq:uncon:arg1_5:1} \\
    &\leq \sigma_{\max} \max_{x \in \x} \|x - x^*\|^2 + 2 \zeta_p \cdot 2^{-2n} \label{eq:uncon:arg1_5:2}\\
    &\leq \sigma_{\max} l \cdot 2^{-2n} + 2 \zeta_p \cdot 2^{-2n} \label{eq:uncon:arg1_5:3}\\
    &\leq \epsilon \label{eq:uncon:arg1_5:4}
\end{align}
where 
\eqref{eq:uncon:arg1_5:1} is a verbatim copy of \eqref{eq:uncon:arg3};
\eqref{eq:uncon:arg1_5:2} follows from \eqref{eq:uncon:approx};
\eqref{eq:uncon:arg1_5:3} follows from Remark \ref{rem:subdiv};
and \eqref{eq:uncon:arg1_5:4} follows from \eqref{eq:uncon:req}.
This implies Algorithm \ref{alg:pcba} terminates no later than iteration $n$.
To conclude the proof, we define
\begin{equation}
    C_1 := -\log_2 \left( \frac{R}{\sqrt{l}} \right), \quad
    C_2 := \frac{1}{2} \log_2 \left( \sigma_{\max} l + 2 \zeta_p \right)
\end{equation}
\end{proof}

\subsection{Theorem \ref{thm:uncon:memory} (Unconstrained Memory Usage)}

\begin{proof}
Since $m < \infty$,
let $\sigma_{\max}$ and $\sigma_{\min}$ be the maximum and minimum singular values, respectively, of $\nabla^2 p(x^*_s)$ over all $s = 1, \cdots, m$.
Notice $\sigma_{\max} \geq \sigma_{\min} > 0$.
Using similar arguments in the proof of Theorem \ref{thm:uncon:roc2}, an analogue of \eqref{eq:uncon:approx} can be obtained.
That is, there exists $R>0$ such that
\begin{equation}
\label{eq:uncon:approx2}
    \frac{1}{4} \sigma_{\min} \|x_s - x_s^*\|^2 \leq p(x_s) - p(x_s^*) \leq \sigma_{\max} \|x_s - x^*_s\|^2
\end{equation}
for all $s = 1, \cdots, m$ and $x_s \in \B_R(x_s^*)$.

Without loss of generality, let $x^*$ represent any minimizer $x_1^*, \cdots, x_m^*$ of \eqref{eq:POP_unconstrained}.
Let $\y \subset \B_R(x^*)$ be a subbox in $\uu$ such that
\begin{equation}
\label{eq:uncon:y}
    \min_{y \in \y} \|y - x^*\| > \sqrt{\frac{4 (\sigma_{\max} l + 2\zeta_p)}{\sigma_{\min}}} \cdot 2^{-n} =: R'.
\end{equation}
We claim that $\y$ cannot be in the list $\LL$.
To prove this, let $\x \subset \B_R(x^*)$ be the subbox that contains $x^*$.
Then
\begin{align}
    \min B(\y) &\geq \min_{y \in \y} p(y) - \zeta_p \cdot 2^{-2n} \label{eq:uncon:arg2:1} \\
    &\geq p(x^*) + \frac{1}{4} \sigma_{\min} \left( \min_{y \in \y} \|y - x^*\|^2 \right) - \zeta_p \cdot 2^{-2n} \label{eq:uncon:arg2:2} \\
    &> p(x^*) + \sigma_{\max} \cdot \left( \sqrt{l} \cdot 2^{-n} \right)^2 + \zeta_p \cdot 2^{-2n} \label{eq:uncon:arg2:3} \\
    &\geq \max_{x \in \B_{\sqrt{l} \cdot 2^{-n}}(x^*)} p(x) + \zeta_p \cdot 2^{-2n} \label{eq:uncon:arg2:4} \\
    &\geq \max_{x \in \x} p(x) + \zeta_p \cdot 2^{-2n} \label{eq:uncon:arg2:5} \\
    &\geq  \max B(\x) \label{eq:uncon:arg2:6}
\end{align}
where 
$\zeta_p$ is the constant in Theorem \ref{thm:err_bound} corresponding to polynomial $p$;
\eqref{eq:uncon:arg2:1} follows from Theorem \ref{thm:err_bound};
\eqref{eq:uncon:arg2:2} follows from \eqref{eq:uncon:approx2};
\eqref{eq:uncon:arg2:3} follows from \eqref{eq:uncon:y} and the strict monotonicity of quadratic functions restricted to non-negative numbers;
\eqref{eq:uncon:arg2:4} follows from \eqref{eq:uncon:approx2};
\eqref{eq:uncon:arg2:5} follows from Remark \ref{rem:subdiv};
\eqref{eq:uncon:arg2:6} follows from Corollary \ref{cor:err_bound}.
According to Theorem \ref{thm:cut-off}, such a subbox $\y$ cannot be in $\LL$.
In other words (by taking the contrapositive), the distance between $x^*$ and any subbox contained in $\B_R(x^*)$ cannot exceed $R'$, meaning $\min_{y \in \y} \norm{y - x^*} \leq R'\ \forall\ \y \in \LL \text{ and } \y \subset \B_R(x^*)$.
Using \eqref{eq:uncon:y}, we can choose a large enough $n$ such that $R' < R$.
Therefore, all subboxes contained in $\B_R(x^*)$ must also be contained in a hypercube centered at $x^*$ with widths $2R'+2\cdot 2^{-n}$. Since the maximum width of all subboxes is $2^{-n}$ (Remark \ref{rem:subdiv}), the number of subboxes contained in $\B_R(x^*)$ is bounded from above by
\begin{equation}
    \left( \frac{2 R' + 2 \cdot 2^{-n}}{2^{-n}} \right)^l 
    = \left( 4 \sqrt{\frac{(\sigma_{\max}l + 2\zeta_p)}{\sigma_{\min}}} + 2 \right)^l =: C_3.
\end{equation}
Since there are $m$ global minimizers, the total number of subboxes remaining in $\LL$ is bounded from above by $m \cdot C_3$, which is a constant.
\end{proof}

\subsection{Theorem \ref{thm:roc:constrained} (Constrained Rate of Convergence)}

\begin{proof}
Let $n \geq 1$ be the current iteration number.
To prove Algorithm \ref{alg:pcba} solves \eqref{eq:POP} in $n$\ts{th} iteration, we must show that there exists a  subbox $\x \in \LL$ that meets the stopping criterion in Definition \ref{def:tolerances}.
That is, $\x$ should satisfy
\begin{enumerate}[label=(\alph*),ref=(\alph*)]
    \item $\x \in \LL$, \label{cond:con:0}
    \item $|\x| \leq \delta$, \label{cond:con:1}
    \item $\max B_{gi}(\x) \leq 0$ for all $i = 1, \cdots, \Nineq$, \label{cond:con:2}
    \item $-\eeq \leq \min B_{hj}(\x) \leq 0 \leq \max B_{hj}(\x) \leq \eeq$ for all $j = 1, \cdots, \Neq$, and \label{cond:con:3}
    \item $\max B_p(\x) - \min B_p(\y) \leq \epsilon$ for all $\y \in \LL$. \label{cond:con:4}
\end{enumerate}
We define such a subbox $\x$ as a function of $n$ using the implicit function theorem, then construct $C_7, C_8$, and $C_9$ using these criteria.

Let $\aha$ ($\aha \leq \Nineq$) be the number of active inequality constraints at $x^*$.
Without loss of generality, suppose $g_1, \cdots, g_\aha$ are active.
Define
\begin{equation}
\label{eq:def:A}
    A :=
    \begin{bmatrix}
        \nabla g_1(x^*) &
        \cdots & \nabla g_{\aha}(x^*) &
        \nabla h_1(x^*) &
        \cdots & \nabla h_{\Neq}(x^*)
    \end{bmatrix}^\top
\end{equation}
It follows from LICQ that $A$ is full rank.
Let $Z$ be a matrix whose columns are a basis for the null space of $A$; that is,
\begin{align}
    Z \in \R^{l \times (l-\aha-\beta)}, \quad Z~\regtext{is full rank}, \quad AZ = 0^{(\aha+\beta) \times (l-\aha-\beta)}.
\end{align}
For convenience, define
\begin{align}
    \tilde A := 
    \begin{bmatrix}
    A \\ Z^\top
    \end{bmatrix}
\end{align}
where $\tilde A$ is a square matrix with full rank.
Define a parameterized system of equations $F: \R^l \times \R \to \R^l$ by
\begin{align}
\label{eq:ift}
    F(a,t) = 
    \begin{bmatrix}
        g_1(a) + (\zeta_{g1} + L_{g1}\sqrt{l}) \cdot t \\
        \vdots \\
        g_\aha (a) + (\zeta_{g\aha} + L_{g\aha}\sqrt{l}) \cdot t \\
        h_1(a) \\
        \vdots \\
        h_\Neq (a)\\
        Z^\top x
    \end{bmatrix} = 0^{l \times 1}
\end{align}
where $\zeta_{gi}$ is the constant in Theorem \ref{thm:err_bound} corresponding to polynomial $g_i$, and $L_{gi}$ is the Lipschitz constant of $g_i$ over $\uu$.
Notice in particular that $F(x^*,0) = 0$.
At $a = x^*$, $t = 0$, the Jacobian of $F$ with respect to $a$ is
\begin{align}
    \nabla_a F(x^*, 0) = \tilde A
\end{align}
which is nonsingular by construction of $Z$.
According to the implicit function theorem, the system \eqref{eq:ift} has a unique solution $x$ for all values of $t$ sufficiently small; that is,
there exists a number $R_1>0$ and a function $\xi: [-R_1,R_1] \to \R^{l}$, such that
\begin{align}
    \xi(0) = & x^* \\
    F(\xi(t),t) = & 0 \quad \forall\ t \in [-R_1, R_1]
\end{align}
In particular, such $R_1$ can be obtained \cite{chang2003analytic} using only the bounds of the polynomials $g_i$ and $h_j$ for all $i = 1, \cdots \aha$, $j = 1, \cdots, \Neq$.
We now let $t := 2^{-n}$.
Since $|t|$ can be arbitrarily small as $n$ is increased, satisfying $|t| \leq R$ requires
\begin{align}
\label{eq:con:req:1}
    n \geq -\log_2 R_1.
\end{align}
In the remainder of this proof we define 
$x := \xi (2^{-n})$; that is,
\begin{align}
\label{eq:ift:x}
    F(x,2^{-n}) = 0,
\end{align}
and denote the subbox that contains $x$ as $\x$.
We next show in Steps 1--5 that
such $\x$ satisfies Conditions \ref{cond:con:1}--\ref{cond:con:4} for $n \geq N$.
The case of $\x \not\in \LL$ is discussed in Step 6.

\emph{Step 1 (Condition \ref{cond:con:1}).}
Notice from Remark \ref{rem:subdiv} that
\begin{align}
    |\x| = 2^{-n} \leq 2^{-N} \leq \delta
\end{align}
therefore Condition \ref{cond:con:1} is satisfied.

\emph{Step 2 (Condition \ref{cond:con:2}).}
Notice that
\begin{align}
    \max B_{gi}(\x) &\leq \max_{x' \in \x} g_i(x') + \zeta_{gi} \cdot 2^{-2n} \label{eq:con:arg:1} \\
    &\leq g_i(x) + \max_{x' \in \x} g_i(x') - g_i(x) + \zeta_{gi} \cdot 2^{-2n} \label{eq:con:arg:2} \\
    &\leq -\zeta_{gi} \cdot 2^{-n} - L_{gi} \sqrt{l} \cdot 2^{-n} + L_{gi} \max_{x' \in \x} \|x' - x\| \ + \label{eq:con:arg:3}\\
    &\quad+\zeta_{gi} \cdot 2^{-2n}\nonumber \\
    &< 0 \label{eq:con:arg:4}
\end{align}
for all $i= 1, \cdots, \aha$, where
\eqref{eq:con:arg:1} follows from Corollary \ref{cor:err_bound};
\eqref{eq:con:arg:2} follows by adding and subtracting $g_i(x)$;
\eqref{eq:con:arg:3} follows from \eqref{eq:ift:x} and the definition of $L_{gi}$;
\eqref{eq:con:arg:4} holds because $2^{-2n} < 2^{-n}$.
Therefore Condition \ref{cond:con:2} is satisfied.

\emph{Step 3 (Condition \ref{cond:con:3}).}
Let $L_{hj}$ be a Lipschitz constant of $h_j$ over $\uu$.
To ensure $\x$ satisfies the equality constraint tolerance $\eeq$, we require the following inequality to hold:
\begin{align}
\label{eq:con:req:2}
    n \geq \max_{j = 1, \cdots, \Neq} \left\{ -\log_2\left( \frac{\eeq}{\zeta_{hj} + L_{hj} \sqrt{l}} \right) \right\}.
\end{align}
Notice
\begin{align}
    \max B_{hj}(\x&) \leq \max_{x' \in \x} h_j(x) + \zeta_{hj} \cdot 2^{-2n} \label{eq:arg:step2:1} \\
    &\leq h_j(x) + L_{hj} \cdot \max_{x' \in \x} \|x' - x\| + \zeta_{hj} \cdot 2^{-2n} \label{eq:arg:step2:2} \\
    &\leq 0 + (L_{hj}\sqrt{l} + \zeta_{hj}) \cdot 2^{-n} \label{eq:arg:step2:3} \\
    &\leq \eeq \label{eq:arg:step2:4}
\end{align}
for all $j = 1, \cdots, \Neq$,
where
\eqref{eq:arg:step2:1} follows from Corollary \ref{cor:err_bound};
\eqref{eq:arg:step2:2} follows from the definition of $L_{hj}$;
\eqref{eq:arg:step2:3} holds because $2^{-2n} \leq 2^{-n}$;
and \eqref{eq:arg:step2:4} follows from \eqref{eq:con:req:2}.
Similarly, we may prove
\begin{align}
\label{eq:arg:step2:5}
    \min B_{hj}(\x) \geq -\eeq
\end{align}
for all $j = 1, \cdots, \Neq$.
It then follows from \eqref{eq:arg:step2:4}, \eqref{eq:arg:step2:5}, Theorem \ref{thm:err_bound}, and Corollary \ref{cor:err_bound} that
\begin{align}
    -\eeq \leq \min B_{hj}(\x) \leq 0 \leq \max B_{hj}(\x) \leq \eeq
\end{align}
for all $j = 1, \cdots, \Neq$.
Therefore Condition \ref{cond:con:3} is satisfied.

\emph{Step 4 (First part of Condition \ref{cond:con:4}).}
It is difficult to directly talk about the relationship between $B_p(\x)$ and $B_p(\y)$ for all $\y \in \LL$.
Instead, in Step 4 we bound the difference between $\max B_p(\x)$ and $p(x^*)$ by $\epsilon / 2$;
then, in Step 5, we bound the gap between $p(x^*)$ and $\min B_p(\y)$ for all $\y \in \LL$ also by $\epsilon / 2$.

Define
\begin{align}
\label{eq:def:b}
    b :=
    \begin{bmatrix}
        \zeta_{g1} + L_{g1} \sqrt{l} \\
        \vdots \\
        \zeta_{g\aha} + L_{g\aha} \sqrt{l} \\
        \rule{0pt}{3ex} 
        0 \\
        \vdots \\
        0
    \end{bmatrix}
    \begin{array}{@{} r @{}}
      \null \\ \phantom{\vdots} \\ \phantom{\vdots} \\
      \left.\begin{array}{@{}c@{}} \null\\ \null \\\null\end{array}\right\} l-\aha
    \end{array}.
\end{align}
By taking the total derivative of \eqref{eq:ift:x} with respect to $t$ at $t=0$ we obtain
\begin{align}
    0 = & \nabla_a F(\xi(0),0) \cdot \nabla_t \xi(0) + \nabla_t F(\xi(0),0) \\
    =& \tilde A \cdot \nabla_t \xi(0) + b
\end{align}
Since $\tilde A$ has full rank, $\nabla_t \xi(0) = - \tilde A^{-1} b$.
Using Taylor's theorem, $x := \xi(2^{-n})$ may be written as
\begin{align}
\label{eq:taylor:new:1}
    x = x^* + \nabla_t \xi (0) \cdot 2^{-n} + o(2^{-n})
\end{align}
where the first term follows from $\xi(0) = x^*$.
For sufficiently small $2^{-n}$, the third term in \eqref{eq:taylor:new:1} is dominated by $2^{-n}$; that is, there exists a number $R_2 > 0$, such that
\begin{align}
\label{eq:con:norm_x}
    \|x - x^*\| \leq \left\| 1^{l \times 1} - \tilde A^{-1} b \right\| \cdot 2^{-n} =: C_4 \cdot 2^{-n}
\end{align}

Let $L_p$ be the Lipschitz constant of $p$ in $\B_{R_2}(x^*)$. Suppose
\begin{align}
\label{eq:con:req:3}
    n \geq -\log_2 \left( \frac{\epsilon}{2 \left(L_p \sqrt{l} + L_p C_4 + \zeta_p \right)} \right),
\end{align}
where $\zeta_p$ is the constant in Theorem \ref{thm:err_bound} corresponding to polynomial $p$.
We then claim $\max B_p(\x) - p(x^*)$ is bounded by $\epsilon/2$.
To see this, notice
\begin{align}
    \max B_p(\x) - p(x^*)
    &\leq \max_{x' \in \x} p(x') - p(x^*) + \zeta_p \cdot 2^{-2n} \label{eq:con:arg2:1} \\
    &\leq \max_{x' \in \x} |p(x') - p(x)| + |p(x) - p(x^*)| + \zeta_p \cdot 2^{-n} \label{eq:con:arg2:2} \\
    &\leq \left( L_p \sqrt{l} + L_p C_4  + \zeta_p \right) \cdot 2^{-n} \label{eq:con:arg2:3} \\
    &\leq \frac{\epsilon}{2} \label{eq:con:arg2:4}
\end{align}
where
\eqref{eq:con:arg2:1} follows from Corollary \ref{cor:err_bound};
\eqref{eq:con:arg2:2} follows from triangle inequality and $2^{-2n} \leq 2^{-n}$;
\eqref{eq:con:arg2:3} follows from \eqref{eq:con:norm_x} and definition of $L_p$;
\eqref{eq:con:arg2:4} follows from \eqref{eq:con:req:2}.

For Condition \ref{cond:con:4} to be satisfied, we need to show that minimum value of every other Bernstein patch $\LL$ is within $\epsilon/2$ of $p(x^*)$;
we do this next.

\emph{Step 5 (Second part of Condition \ref{cond:con:4}).}
In this step, we bound the difference between $p(x^*)$ and $\min B_p(\y)$ for all $\y \in \LL$.
To do this, we first show that any $\y \in \LL$ solves a relaxed POP related to the original POP and parameterized by the current iteration $n$.
Then, we bound the difference between the solutions of the relaxed and original POPs.

Let $\y \in \LL$ be any subbox with maximum width $2^{-n}$.
Then,
\begin{enumerate}[label=(\alph*),ref=(\alph*)]
\setcounter{enumi}{5}
    \item $\min B_{gi}(\y) \leq 0$ for all $i = 1, \cdots, \Nineq$; \label{eq:cons:step4:a}
    \item $\min B_{hj}(\y) \leq 0$ for all $j = 1, \cdots, \Neq$; \label{eq:cons:step4:b}
    \item $\max B_{hj}(\y) \geq 0$ for all $j = 1, \cdots, \Neq$, \label{eq:cons:step4:c}
\end{enumerate}
from Theorem \ref{thm:cut-off}.
Notice \ref{eq:cons:step4:a} implies
\begin{align}
    0 &\geq \min_{y \in \y} g_i(y) - \zeta_{gi} \cdot 2^{-2n} \label{eq:cons:step4:arg1} \\
    &\geq \max_{y \in \y} g_i(y) - L_{gi} \sqrt{l} \cdot 2^{-n} - \zeta_{gi} \cdot 2^{-2n} \label{eq:cons:step4:arg2}
\end{align}
where \eqref{eq:cons:step4:arg1} follows from Theorem \ref{thm:err_bound};
\eqref{eq:cons:step4:arg2} follows from definition of $L_{gi}$.
For convenience, denote 
\begin{align}
    \gamma_{gi}(n) &:= L_{gi} \sqrt{l} \cdot 2^{-n} + \zeta_{gi} \cdot 2^{-2n}.
\end{align}
Then, \eqref{eq:cons:step4:arg2} can be written as
\begin{equation}
\label{eq:cons:_1}
    \max_{y \in \y} g_i(y) \leq \gamma_{gi}(n).
\end{equation}
Similarly, if we define
\begin{align}
    \gamma_{hj}(n) &:= L_{hj} \sqrt{l} \cdot 2^{-n} + \zeta_{hj} \cdot 2^{-2n},
\end{align}
then \ref{eq:cons:step4:b} and \ref{eq:cons:step4:c} can be written as
\begin{align}
    \max_{y \in \y} h_j(y) \leq \gamma_{hj}(n) \label{eq:cons:_2} \\
    \min_{y \in \y} h_j(y) \geq -\gamma_{hj}(n). \label{eq:cons:_3}
\end{align}

Now we are ready to introduce the relaxed POP:
\begin{align}
         \underset{y \in \uu}{\min} & \quad p(y) & \tag{$\regtext{P}_n$}\label{prog:relaxed_POP} \\
         \regtext{s.t} & \quad g_i(y) \leq \gamma_{gi}(n) & i = 1,\cdots,\,\Nineq \label{eq:cons:relaxed:1}\\
                       & \quad h_j(y) \leq \gamma_{hj}(n) & j = 1,\cdots,\,\Neq \label{eq:cons:relaxed:2}\\
                       & \quad h_j(y) \geq -\gamma_{hj}(n) & j = 1,\cdots,\,\Neq \label{eq:cons:relaxed:3}
\end{align}
Denote $y_n^*$ as the optimal solution to \eqref{prog:relaxed_POP}.
Notice that \eqref{eq:cons:relaxed:1}, \eqref{eq:cons:relaxed:2}, and \eqref{eq:cons:relaxed:3} are relaxations of \eqref{eq:cons:_1}, \eqref{eq:cons:_2}, and \eqref{eq:cons:_3} respectively.
That is, let $\y \in \LL$ and consider any $y \in \y$ that satisfies \eqref{eq:cons:_1}, \eqref{eq:cons:_2}, and \eqref{eq:cons:_3}.
It also necessarily satisfies \eqref{eq:cons:relaxed:1}, \eqref{eq:cons:relaxed:2}, and \eqref{eq:cons:relaxed:3}.
Therefore,
\begin{equation}
    p(y_n^*) \leq \min_{y \in \y} p(z)
\end{equation}
It can be obtained from Theorem \ref{thm:err_bound} that 
\begin{equation}
\label{eq:cons:step4:B&z*}
    \min_{\y \in \LL} B_p(\y) \geq \min_{y \in \y} p(y) - \zeta_p \cdot 2^{-2n} \geq p(y_n^*) - \zeta_p \cdot 2^{-2n}
\end{equation}

To estimate the value $p(y^*_n)$, notice from \cite[Corollary~4.1]{fiacco1978nonlinear},
\begin{align}
    \frac{\partial p(y_n^*)}{\partial \gamma_{gi}(n)} = & -\lambda_{gi}^* \\
    \frac{\partial p(y_n^*)}{\partial \gamma_{hj}(n)} = & -\lambda_{hj}^*
\end{align}
where $\lambda^*_{gi}$ and $\lambda^*_{hj}$ are Lagrange multipliers corresponding to the constraints $g_i$ and $h_j$ in \eqref{prog:relaxed_POP}, respectively, for all $i = 1, \cdots, \, \aha$ and $j = 1, \cdots, \, \Neq$.
Using Taylor's theorem, we obtain
\begin{align}
    \begin{split}
        p(y^*_n) \geq\,&p(x^*) - \sum_{i = 1, \cdots, \,\Nineq} \lambda^*_{gi} \gamma_{gi}(n) - \sum_{j = 1, \cdots, \,\Neq} \left| \lambda^*_{hj} \gamma_{hj}(n) \right| +  \\
                 &+ O \left(\sum_{i = 1, \cdots, \,\Nineq} \gamma_{gi}^2(n) + \sum_{j = 1, \cdots, \,\Neq} \gamma_{hj}^2(n) \right),
    \end{split}
\end{align}
for sufficiently small $\gamma_{gi}(n)$ and $\gamma_{hj}(n)$.

Since $\gamma_{gi}(n)$ and $\gamma_{hj}(n)$ can be made arbitrarily small by increasing $n$, there exists a constant $C_5$ such that
\begin{equation}
\label{eq:con:step4:taylor}
    p(y^*_n) \geq p(x^*) - \frac{1}{2} \left( \sum_{i = 1, \cdots, \,\Nineq} \lambda^*_{gi} \gamma_{gi}(n) + \sum_{j = 1, \cdots, \,\Neq} 
    \left|\lambda^*_{hj} \gamma_{hj}(n)\right| \right)
\end{equation}
for all $n \geq C_5$.
It then follows from \eqref{eq:cons:step4:B&z*} and \eqref{eq:con:step4:taylor} that
\begin{align}
    p(x^*) - \min_{\y \in \LL} B_p(\y) \leq & \frac{1}{2} \left( \sum_{i = 1, \cdots, \,\Nineq} \lambda^*_{gi} \gamma_{gi}(n) + \sum_{j = 1, \cdots, \,\Neq} \left| \lambda^*_{hj} \gamma_{hj}(n) \right| \right) + \zeta_p \cdot 2^{-2n}  \label{eq:cons:step4:p_minus_B}
\end{align}
Suppose $n$ satisfies
\begin{align}
\label{eq:con:req:4}
    n \geq -\log_2 \epsilon + \log_2 ( C_6 + 2 \zeta_p )
\end{align}
where
\begin{align}
\label{eq:con:def:c6}
    C_6 := \sum\limits_{i = 1, \cdots, \,\Nineq} \lambda^*_{gi} (L_{gi} \sqrt{l} + \zeta_{gi}) + \sum\limits_{j = 1, \cdots, \,\Neq} \left| \lambda^*_{hj} (L_{hj} \sqrt{l} + \zeta_{hj}) \right|
\end{align}
Then it follows from \eqref{eq:cons:step4:p_minus_B} that
\begin{align}
    \label{eq:cons:step4:p_minus_b_less_than_ep_over_2}
    p(x^*) - \min_{\y \in \LL} B_p(\y) \leq \frac{\epsilon}{2}.
\end{align}
By applying the triangle inequality to \eqref{eq:con:arg2:4} and \eqref{eq:cons:step4:p_minus_b_less_than_ep_over_2}, we fulfill Condition \ref{cond:con:4}.

\emph{Step 6 (Condition \ref{cond:con:0}).}
If $\x \in \LL$ in the above discussion, then such $\x$ satisfies Conditions \ref{cond:con:0}--\ref{cond:con:4} and therefore the proof is finished.
When $\x \not\in \LL$, however, we need to show there exists another subbox $\z \subset \uu$ that satisfies Conditions \ref{cond:con:0}--\ref{cond:con:4}.
This can be done using the Cut-Off Test in Theorem \ref{thm:cut-off}.

Suppose $\x \not\in \LL$.
Notice such $\x$ is \emph{not infeasible} per Definition \ref{def:feasible}.
According to Theorem \ref{thm:cut-off}, such $\x$ must be \emph{suboptimal} so that it is removed from the list $\LL$; that is, there exists a feasible $\z \in \LL$ such that
\begin{align}
\label{eq:p(z)_p(x)}
    \max B_p(\z) < \min B_p(\x)
\end{align}
Per Definition \ref{def:feasible}, such $\z$ satisfies Conditions \ref{cond:con:0}--\ref{cond:con:3}.
To show $\z$ satisfies Condition \ref{cond:con:4}, notice from \eqref{eq:con:arg2:4}, \eqref{eq:cons:step4:p_minus_b_less_than_ep_over_2},
and \eqref{eq:p(z)_p(x)} that
\begin{align}
    \max B_p(\z) - \min B_p(\y) <  \max B_p(\x) - \min B_p(\y) \leq \epsilon
\end{align}
for all $\y \in \LL$.

To conclude the proof, define constants
\begin{align}
    C_7 := & \max \{-\log_2 R_1, C_5\} \\
    C_8 := & \max_{j = 1, \cdots, \Neq} \log_2 \left( \zeta_{hj} + L_{hj} \sqrt{l} \right) \\
    C_9 := & \max \left\{ \log_2 \left( 2 \left( L_p \sqrt{l} + L_p C_4 + \zeta_p \right) \right), \log_2 (C_6 + 2 \zeta_p) \right\}
\end{align}
The result then follows.
\end{proof}

\subsection{Theorem \ref{thm:memory_usage:constrained} (Constrained Memory Usage)}

\begin{proof}
Consider an arbitrary global minimizer $x^* \in \uu$.
It follows from \eqref{eq:con:arg2:3} (in the proof of Theorem \ref{thm:roc:constrained}) that for sufficiently large $n$ there exists a subbox $\x \in \LL$ such that
\begin{align}
\label{eq:con:mem:bound_x_x*}
    \max B_p(\x) - p(x^*) \leq \left( L_p\sqrt{l} + L_p C_4 + \zeta_p \right) \cdot 2^{-n}
\end{align}
where $C_4$ is defined as in \eqref{eq:con:norm_x}.
Consider another feasible point $z \in \uu$ of \eqref{eq:POP}, and let $\z \subset \uu$ be the subbox that contains $z$.
Such $z$ is necessarily feasible to the relaxed problem \eqref{prog:relaxed_POP}.
According to \cite[(12.71)]{nocedal2006numerical}, we know
\begin{align}
\label{eq:con:mem:taylor:1}
    p(z) \geq p(y_n^*) + \|z - y_n^*\| \sum_{i = 1, \cdots, \aha} \lambda_{gi}^* \nabla g_i(y_n^*)^\top d + o(\|z - y_n^*\|)
\end{align}
where 
$y_n^*$ is the optimal solution to \eqref{prog:relaxed_POP};
$\lambda_{gi}^*$ are the Lagrange multipliers corresponding to $g_i$ in \eqref{prog:relaxed_POP} for all $i = 1, \cdots, \aha$;
and $d$ is any unit vector such that 
\begin{align}
    & \nabla g_i(y_n^*)^\top d \geq 0 \quad \forall\ i = 1, \cdots, \aha, \label{eq:def:d:1} \\
    & \nabla h_j(y_n^*)^\top d = 0 \quad \forall\ j = 1, \cdots, \Neq. \label{eq:def:d:2}
\end{align}
Since the critical cone of \eqref{prog:relaxed_POP} at $y_n^*$ is nonempty, for any $d$ there exists $i \in \{1, \cdots, \aha \}$ such that
\begin{align}
\label{eq:d:pos}
    \lambda_{gi}^* \nabla g_i(y_n^*)^\top d > 0
\end{align}
For notational convenience, define
\begin{align}
    J(d) = \sum_{i = 1, \cdots, \aha} \lambda_{gi}^* \nabla g_i(y_n^*)^\top d
\end{align}
with minimum $q := \min_d J(d)$.
In particular, it follows from the extreme value theorem that such minimum exists, since $J(d)$ is continuous with respect to $d$, and the set of unit vectors defined by \eqref{eq:def:d:1} and \eqref{eq:def:d:2} is closed and bounded (therefore compact according to Heine-Borel theorem).
It also follows from \eqref{eq:d:pos} that $q > 0$.
For any $z$ that is sufficiently close to $y_n^*$, 
we may then rewrite \eqref{eq:con:mem:taylor:1} as
\begin{align}
\label{eq:con:mem:q}
    p(z) \geq p(y_n^*) + \frac{q}{2} \|z - y_n^*\|
\end{align}

We claim for sufficiently large n and all $z \in \uu$ such that
\begin{align}
\label{eq:con:mem:def:c10}
    \|z - y_n^*\| \geq \frac{2\left( 2 L_p \sqrt{l} + L_p C_4 + C_6/2 + 2 \zeta_p \right)}{q} \cdot 2^{-n} =: C_{10} \cdot 2^{-n},
\end{align}
the subbox $\z$ cannot be in the list $\LL$. To prove this, notice
\begin{align}
    \min B(\z) \geq & \min_{z' \in \z} p(z') - \zeta_p \cdot 2^{-2n} \label{eq:con:mem:arg:1} \\
    \geq & p(z) - L_p \sqrt{l} \cdot 2^{-n} - \zeta_p \cdot 2^{-2n} \label{eq:con:mem:arg:2} \\
    \geq & p(y_n^*) + \left( \frac{q C_{10}}{2}- L_p \sqrt{l} - \zeta_p\right) \cdot 2^{-n} \label{eq:con:mem:arg:3} \\
    \geq & p(x^*) + \left( \frac{q C_{10}}{2}- L_p \sqrt{l} - \frac{C_6}{2} - \zeta_p \right) \cdot 2^{-n} \label{eq:con:mem:arg:4} \\
    \geq & \max B_p(\x) + \left( \frac{q C_{10}}{2} - 2 L_p \sqrt{l} - L_p C_4 - \frac{C_6}{2} - 2 \zeta_p \right) \cdot 2^{-n} \label{eq:con:mem:arg:5} \\
    \geq & \max B_p(\x) \label{eq:con:mem:arg:6} 
\end{align}
where
$L_p$ is the Lipchitz constant of $p$ over $\uu$,
and $C_6$ is defined as in \eqref{eq:con:def:c6}.
In particular,
\eqref{eq:con:mem:arg:1} follows from Theorem \ref{thm:err_bound};
\eqref{eq:con:mem:arg:2} follows from definition of $L_p$;
\eqref{eq:con:mem:arg:3} follows from \eqref{eq:con:mem:q};
\eqref{eq:con:mem:arg:4} follows from \eqref{eq:con:step4:taylor};
\eqref{eq:con:mem:arg:5} follows from \eqref{eq:con:mem:bound_x_x*};
\eqref{eq:con:mem:arg:6} follows from \eqref{eq:con:mem:def:c10}.
According to Theorem \ref{thm:cut-off}, $\z$ cannot be in the list $\LL$.

To conclude this proof, notice the above claim implies that for sufficiently large $n$, 
all subboxes in a neighborhood of $x^*$ must also be contained in a hypercube centered at $y_n^*$ with widths $(2 C_{10} + 2) \cdot 2^{-n}$.
Since the maximum width of all subboxes is $2^{-n}$ (Remark \ref{rem:subdiv}), 
the number of subboxes contained in a neighborhood of $x^*$ is bounded from above by
\begin{align}
    \left( \frac{(2 C_{10} + 2) \cdot 2^{-n}}{2^{-n}} \right)^l
    = \left( 2 C_{10} + 2 \right)^l.
\end{align}
Since there are $m$ global minimizers, the number of subboxes remaining in $\LL$ is bounded above by $m \cdot (2 C_{10} + 2)^l$, which is a constant
\end{proof}
\section{Objective Functions}\label{app:test_functions}

\subsection{Benchmark Problems}
The following problems are all from \cite{nataraj2011constrained}, with minor changes.
To report the global optima values and locations, we solved each problem 1000 times using MATLAB's \fmincon, with optimality and constraint tolerances set to $10^{-10}$, and random initial guesses.

\textbf{P1}:
\begin{align*}
\min\ &-x_1-x_2\\
\regtext{s.t.}\hspace{0.4em} &-2x_1^4+8x_1^3-8x_1^2+x_2-2 \leq 0 \\
                  &-4x_1^4+32x_1^3-88x_1^2+96x_1+x_2-36 \leq 0
\end{align*}
where $x_1\in[0,3]$ and $x_2\in[0,4]$.
The best optimal value we found is $-5.5080132636$ at the location $x = (2.3295201981,3.1784930655)$.

\textbf{P2}:
\begin{align*}
\min\ &658500x_1^3 + 68121x_1^2 + 2349x_1 + \\
    &+1000000x_2^3 -600000x_2^2 + 120000x_2 - 7973 \\
\regtext{s.t.}\hspace{0.4em} &-7569x_1^2 - 1392x_1 - 10000x_2^2 + 1000x_2 + 11 \leq 0 \\
                  &7569x_1^2 + 1218x_1 + 10000x_2^2 - 1000x_2 - 8.81 \leq 0
\end{align*}
where $x_1\in[0,1]$ and $x_2\in[0,1]$.
The best optimal value we found is $-6961.8138816446$ at $x = (0.0125862069,0.0084296079)$.

\textbf{P3}:
\begin{align*}
\min\ &x_1\\
\regtext{s.t.}\hspace{0.4em} &x_1^2-x_2 \leq 0 \\
                  &x_2-x_1^2(x_1-2)+10^{-5} \leq 0
\end{align*}
where $x_i\in[-10,10]$ for $i=1,2$.
The best optimal value we found is $3.0000011115$ at $x = (3.0000011115,9.0000066709)$.

\textbf{P4}:
\begin{align*}
\min\ &-2x_1+x_2-x_3\\
\regtext{s.t.}\hspace{0.4em} &x_1+x_2+x_3-4 \leq 0 \\
                  &3x_2+x_3-6 \leq 0 \\
                  &-{x}^T{A}^T{A}{x}+2{y}^T{A}{x}-\left\lVert{y}\right\rVert^2+0.25\left\lVert{b-z}\right\rVert^2 \leq 0 \\
                  &{A}=\left(\begin{array}{ccc}
                  0&0&1\\
                  0&-1&0 \\
                  -2&1&-1
                  \end{array}\right)\\
                  &{b}=(3,0,-4)^T\\
                  &{y}=(1.5,-0.5,-5)^T\\
                  &{z}=(0,-1,-6)^T
\end{align*}
where $x_1\in[0,2]$, $x_2\in[0,10]$, and $x_3\in[0,3]$.
The global optimum is $-4$ at $x = (0.5,0,3)$.

\textbf{P5}:
\begin{align*}
\min\ &x_3\\
\regtext{s.t.}\hspace{0.4em} &2x_1^2+4x_1x_2-42x_1+4x_1^3-x_3-14 \leq 0 \\
                  &-2x_1^2-4x_1x_2+42x_1-4x_1^3-x_3+14 \leq 0 \\
                   &2x_1^2+4x_1x_2-26x_1+4x_1^3-x_3-22 \leq 0 \\
                  &-2x_1^2-4x_1x_2+26x_1-4x_1^3-x_3+22 \leq 0
\end{align*}
where $x_i\in[-5,5]$ for $i=1,2,3$.
The best optimum found is $0$ at $x = (-0.3050690380,-0.9133455177,0)$.

\textbf{P6}:
\begin{align*}
\min\ &0.6224x_3x_4+1.7781x_2x_3^2+3.1661x_1^2x_4+19.84x_1x_3\\
\regtext{s.t.}\hspace{0.4em} &-x_1+0.0193x_3 \leq 0 \\
                   &-x_2+0.00954x_3 \leq 0 \\
                   &-\pi x_3^2x_4-\frac{4}{3}\pi x_3^3+750.1728 \leq  0 \\
                   &-240+x_4 \leq 0
\end{align*}
where $x_1\in[1,1.375]$, $x_2\in[0.625,1]$, $x_3\in[47.5,52.5]$, and $x_4\in[90,112]$.
The global optimum is $6395.5078$ (to the same number of decimal places as used in the cost function) at the location $x = (1,0.625,47.5,90.0)$ ;

\textbf{P7}:
\begin{align*}
\min\ &x_4\\
\regtext{s.t.}\hspace{0.4em} &x_1^4x_2^4-x_1^4-x_2^4x_3 = 0 \\
                   &1.4-0.25x_4-x_1\leq0 \\
                   &-1.4-0.25x_4+x_1\leq0 \\
                   &1.5-0.2x_4-x_2\leq0 \\
                   &-1.5-0.2x_4+x_2\leq0 \\
                   &0.8-0.2x_4-x_3\leq0 \\
                   &-0.8-0.2x_4+x_3\leq0
\end{align*}
where $x_i\in[0,5]$ for $i=1,2,3,4$.
The best optimal value we found is $1.0898639714$ at the location $x = (1.1275340071,1.2820272057,1.0179727943,1.0898639714)$.

$\textbf{P8}$:
\begin{align*}
\min\ &54.528x_2x_4+27.624x_1x_3-54.528x_3x_4\\
\regtext{s.t.}\hspace{0.4em} &61.01627586-I \leq 0 \\
                   &8x_1 - I(x) \leq 0 \\
                   &x_1x_2x_4-x_2x_4^2+x_1^2x_3+x_3x_4^2-2x_1x_3x_4-3.5x_3I(x) \leq 0 \\
                   &x_1 - 3x_2 \leq 0 \\
                   &2x_2-x_1 \leq 0 \\
                   &x_3-1.5x_4 \leq 0 \\
                   &0.5x_4-x_3 \leq 0 \\
                   &I(x) = 6x_1^2x_2x_3-12x_1x_2x_3^2+8x_2x_3^3+x_1^3x_4-6x_1^2x_3x_4+ \\
                   &\hspace{1cm}+12x_1x_3^2x_4-8x_3^3x_4
\end{align*}
where $x_1\in[3,20]$, $x_2\in[2,15]$, $x_3\in[0.125,0.75]$, and $x_4\in[0.25,1.25]$.
The best optimal value we found is $42.4440570797$ at the location $x = (4.9542421008,2,0.125,0.25)$.

\subsection{Increasing Constraints Problems}

\textbf{ElAttar-Vidyasagar-Dutta} (E-V-D) \cite{gavana2019testsuite}:
\begin{align*}
    (x_1^2+x_2-10)^2+(x_1+x_2^2-7)^2+(x_1^2+x_2^3-1)^2
\end{align*}
where $x_i\in[-100,100]$ for $i = 1, 2$.
The global optimum is $1.712780354$ at $x = (3.40918683,-2.17143304)$.

\textbf{Powell} \cite{gavana2019testsuite}:
\begin{align*}
    (x_1+10x_2)^2+5(x_3-x_4)^2+(x_2-2x_3)^4+10(x_1-x_4)^4
\end{align*}
where $x_i\in[-10,10]$ for $i = 1, 2, 3, 4$.
The global optimum is $0$ at $x = 0$.

\textbf{Wood} \cite{gavana2019testsuite}:
\begin{align*}
    (100(x_2-x_1^2))^2+(1-x_1)^2+90(x_4-x_3^2)^2+(1-x_3)^2+ \\
    +10.1((x_2-1)^2+(x_4-1)^2)+19.8(x_2-1)(x_4-1)
\end{align*}
where $x_i\in[-10,10]$ for $i=1,2,3,4$.
The global optimum is $0$ at $x = (1,1,1,1)$.

\textbf{Dixon-Price} (D-P) \cite{gavana2019testsuite}:
\begin{align*}
    (x_1-1)^2+\sum_{i=2}^di(2x_i^2-x_{i-1})^2
\end{align*}
where $x_i\in[-10,10]$ for $i = 1,\ldots,d$.
The global optimum is $0$ at $x_i=2^{-\frac{2^i-2}{2^i}}$ for $i=1,\ldots,d$.

\textbf{Beale}:
\begin{align*}
    (x_1x_2-x_1+1.5)^2+(x_1x_2^2-x_1+2.25)^2+\\
    +(x_1x_2^3-x_1+2.625)^2
\end{align*}
where $x_i\in[-10,10]$ for $i=1,2$.
The global optimum is $0$ at $x = (3,0.5)$.

\textbf{Bukin02} \cite{gavana2019testsuite}:
\begin{align*}
    100(x_2^2-0.01x_1^2+1)+0.01(x_1+10)^2
\end{align*}
where $x_1\in[-15,-5]$, $x_2\in[-3,3]$.
The global optimum is $-124.75$ at $x = (-15,0)$.

\textbf{Deckkers-Aarts} (D-A):
\begin{align*}
    10^5x_1^2+x_2^2-(x_1^2+x_2^2)^2+10^{-5}(x_1^2+x_2^2)^4
\end{align*}
where $x_i\in[-20,20]$ for $i=1,2$.
The global optimum is $-24777$ at $x = (0,\pm15)$.
\end{appendices}

\end{document}